\documentclass{siamart1116} 
\usepackage{graphicx}
\usepackage{xcolor}
\usepackage{soul}
\usepackage{amsmath}
\usepackage{amssymb}
\usepackage{tikz}
\usetikzlibrary{arrows}
\usetikzlibrary{calc}
\usetikzlibrary{decorations.pathmorphing}
\usepackage{algorithmic}

\newcommand{\mO}{\mathcal{O}}
\newcommand{\mJ}{\mathcal{I}}
\newcommand{\mI}{\mathcal{J}}

\newtheorem{remark}[theorem]{Remark}
\newtheorem{example}[theorem]{Example}

\numberwithin{theorem}{section}

\newcommand{\TheTitle}{The cut-off resolvent can grow arbitrarily fast in obstacle scattering}
\newcommand{\TheAuthors}{S.~N. Chandler-Wilde and S. Sadeghi}

\headers{Arbitrarily fast resolvent growth in obstacle scattering}{\TheAuthors}

\title{{\TheTitle}\thanks{Submitted to the editors 23 August 2025.
\funding{This work was funded by the UK Engineering and Physical Sciences Research Council (EPSRC) through a PhD Studentship for the second author.}}}

\author{ Simon N. Chandler-Wilde\thanks{Department of Mathematics and Statistics, University of Reading, Whiteknights PO Box 220, Reading RG6 6AX, UK
    (\email{s.n.chandler-wilde@reading.ac.uk}, \email{s.sadeghi@pgr.reading.ac.uk})}
    \and Siavash Sadeghi\footnotemark[2]
}

\ifpdf
\hypersetup{
  pdftitle={\TheTitle},
  pdfauthor={\TheAuthors}
}
\fi

\begin{document}
\newcommand{\R}{\mathbb{R}}
\newcommand{\N}{\mathbb{N}}
\newcommand{\C}{\mathbb{C}}
\newcommand{\Hol}{H^{1,\mathrm{loc}}}
\newcommand{\Lloc}{L^{2}_{\mathrm{loc}}}
\newcommand{\Lcomp}{L^{2}_{\mathrm{comp}}}
\newcommand{\Hmc}{H^{-1}_{\mathrm{comp}}}
\newcommand{\ri}{\mathrm{i}}
\newcommand{\cS}{\mathcal{S}}
\newcommand{\cA}{\mathcal{A}}
\newcommand{\cH}{\mathcal{H}}
\newcommand{\cF}{\mathcal{F}}
\newcommand{\supp}{\mathrm{supp}}
\newcommand{\spec}{\mathrm{spec}}
\newcommand{\cond}{\mathrm{cond}}
\newcommand{\rd}{\mathrm{d}}
\newcommand{\re}{\mathrm{e}}
\newcommand{\dist}{\mathrm{dist}}
\newcommand{\rea}{\mathrm{Re}\,}

\newcommand{\tr}[1]{\text{tr}_{#1}}

\newcommand{\Tr}[1]{\text{Tr}_{#1}}
\newcommand{\DTN}{\mathrm{DtN}_R}

\maketitle

\begin{abstract}
We consider time-harmonic acoustic scattering by a compact sound-soft obstacle $\Gamma\subset \R^n$ ($n\geq 2$) that has connected complement  $\Omega := \R^n\setminus \Gamma$. This scattering problem is modelled  by the inhomogeneous Helmholtz equation $\Delta u + k^2 u = -f$  in $\Omega$, the boundary condition that $u=0$ on $\partial \Omega = \partial \Gamma$, and the standard Sommerfeld radiation condition. It is well-known that, if the boundary $\partial \Omega$ is smooth, then the norm of the cut-off resolvent of the Laplacian, that maps the compactly supported inhomogeneous term $f$ to the solution $u$ restricted to some ball, grows at worst exponentially with $k$. In this paper we show that, if no smoothness of $\Gamma$ is imposed, then the growth can be arbitrarily fast. Precisely, given some modestly increasing unbounded sequence $0<k_1<k_2<\ldots$ and some arbitrarily rapidly increasing sequence  $0<a_1<a_2<\ldots$, we construct a compact $\Gamma$ such that, for each $j\in \N$, the norm of the cut-off resolvent at $k=k_j$ is greater than $a_j$.
\end{abstract}

\begin{keywords}
Helmholtz equation, scattering, trapping, resolvent
\end{keywords}

\begin{AMS}
35J05, 35J25, 35P25, 47A10, 78A45
\end{AMS}


\section{Introduction} \label{sec:intro}
Given some compact set $\Gamma\subset \mathbb{R}^n$ ($n\geq 2$) such that $\Omega:= \mathbb{R}^n\setminus \Gamma$ is connected, and some compactly supported source term $f\in \Lcomp(\Omega)$, a classic problem in scattering theory is to seek $u\in \Hol_0(\Delta, \Omega)$ 
that is a solution  of the inhomogeneous Helmholtz equation
\begin{equation}\label{eq:he}
\Delta u+k^2 u=-f,
\end{equation}
in a weak sense in $\Omega$, 
and that satisfies the
Sommerfeld radiation condition
\begin{equation}\label{eq:src}
 \frac{\partial u}{\partial r}(x)-\ri ku(x)=o\left(r^{-\frac{n-1}{2}}\right), \quad \text{ uniformly in }\frac{x}{r}, \quad \mbox{ as } r:=\lvert x\rvert\rightarrow \infty.
\end{equation}
In the above equations $k>0$ is the given wavenumber, and the function space notations we use are recalled in \S\ref{sec:fs}.
It is well known that this scattering problem, which imposes on $\partial \Omega$ a zero boundary condition on $u$ through the requirement that $u\in \Hol_0(\Delta, \Omega)$, is well-posed. Indeed, uniqueness can be proved via Green's theorem and Rellich's lemma, and existence and continuous dependence on data is proved, e.g., by {\em a priori} estimates coupled with a limiting absorption argument (see, e.g., \cite{wilcox1975,NR87,mclean2000strongly}). 

\subsection{The question we address} 
It follows from the well-posedness of the above scattering problem that, given any $R>R_\Gamma:=\max_{x\in \Gamma}|x|$ and any $k>0$, there exists some minimal $C_{k,R}>0$ such that
\begin{equation} \label{eq:resest}
\|u\|_{L^2(\Omega_R)} \leq C_{k,R} \|f\|_{L^2(\Omega_R)}, 
\end{equation}
whenever $\supp(f)\subset \Omega_R:=\{x\in \Omega:|x|<R\}$. This paper is concerned with answering the following question: 
\begin{equation} 
\label{eq:Q}
\mbox{For fixed $R>R_\Gamma$, how fast can $C_{k,R}$ grow as $k\to\infty$?}
\end{equation}

Let us make two comments regarding \eqref{eq:Q}. First, note that this question is commonly equivalently formulated in terms of norms of the resolvent of the Laplacian. 
By the above well-posedness there is, for each $k>0$,  a well-defined resolvent operator
\begin{equation}\label{resolvent}
R(k):=(-\Delta-k^2)^{-1}:\Lcomp(\Omega)\rightarrow \Hol_0(\Delta,\Omega),
\end{equation}
that maps $f\in \Lcomp(\Omega)$  to the solution $u\in \Hol_0(\Delta,\Omega)$ of the above scattering problem. For $R>R_\Gamma$ let $\chi_R$ denote the characteristic function of $\Omega_R$. Then, for $k>0$ and $R>R_\Gamma$, the minimal $C_{k,R}$ in \eqref{eq:resest} is precisely
\begin{equation} \label{eq:cutoff}
C_{k,R} :=  \|\chi_R R(k)\chi_R\|_{L^2(\Omega)\to L^2(\Omega)},
\end{equation}
the norm of the so-called {\em cut-off resolvent}  $\chi_R R(k)\chi_R: L^2(\Omega)\rightarrow L^2(\Omega)$. 

Second, a natural variant to \eqref{eq:Q} is to ask how the constant $C_{k,R}$ changes if the $L^2(\Omega_R)$ norm of $u$ in \eqref{eq:resest} is replaced by an $H^1(\Omega_R)$ norm, in particular by the natural wavenumber-explicit energy norm $\|\cdot\|_{H_k^1(\Omega_R)}$, defined, for $k>0$, by
\begin{equation} \label{eq:H1k}
\|v\|^2_{H_k^1(\Omega_R)} = \| |\nabla v| \|_{L^2(\Omega_R)}^2 + k^2 \|v\|_{L^2(\Omega_R)}^2, \quad v\in H^1(\Omega_R).
\end{equation}
In other words, what is the minimal $C_{k,R}'>0$ such that
\begin{equation} \label{eq:resest2}
\|u\|_{H^1_k(\Omega_R)} \leq C'_{k,R} \|f\|_{L^2(\Omega_R)}, 
\end{equation}
whenever $\supp(f)\subset \Omega_R$, i.e., what is 
\begin{equation} \label{eq:Cdashdef}
C'_{k,R} :=  \|\chi_R R(k)\chi_R\|_{L^2(\Omega)\to H^1_k(\Omega_R)} \;?
\end{equation} 
This question is equivalent up to a numerical factor to \eqref{eq:Q}: in particular the dependence on $k$ of $C_{k,R}'$ is evident once that of $C_{k,R}$ is known (and vice versa). For  clearly $C_{k,R}'\geq kC_{k,R}$. Moreover, a simple Green's theorem argument, coupled with an estimate for a boundary integral that is well-known in dimension $n=2,3$, and that we establish for all $n\geq 2$ in Lemma \ref{lem:2.1} in the appendix, establishes that
\begin{equation} \label{eq:equiv}
kC_{k,R} \leq C_{k,R}' \leq \sqrt{2k^2C_{k,R}^2 + C_{k,R}}, \qquad k>0, \;\; R>R_\Gamma
\end{equation}  
(see \S\ref{sec:weak} below).

\subsection{Discussion of previous work} \label{sec:related}
There has been, over many decades, large research interest in \eqref{eq:Q}, in particular in teasing out the behaviour of $C_{k,R}$ as a function of $k$ in the large $k$ limit, and how this behaviour depends on the geometry of $\Gamma$. This research has been surveyed recently in \cite{chandler2020high} (see, e.g., \cite[\S1.1, Table 6.1]{chandler2020high}), or see \cite{DyZw}, especially \cite[\S6.6]{DyZw}.

The earliest upper bounds on $C_{k,R}$ as a function of $k$, due to Morawetz and co-workers (see \cite{morawetzludwig1968,morawetz1975}), assume that $\Gamma = \overline{\Omega_-}$, where $\Omega_-$ is some bounded smooth domain, and that $\Gamma$ is star-shaped in the following sense.
\begin{definition}[Star-shaped] \label{def:ss}
We say that a set $T\subset \R^n$ is {\em star-shaped with respect to $x\in T$} if the line segment $[x,y]:= \{tx+(1-t)y:0\leq t\leq 1\}\subset T$ for every $y\in T$. We say that $T$ is {\em star-shaped} if it is star-shaped with respect to some $x\in T$.
\end{definition}
Using similar integration by parts and multiplier arguments to \cite{morawetzludwig1968,morawetz1975}, coupled with properties of radiating solutions recalled in Lemma \ref{lem:2.1} in the appendix, and arguments based on approximation of arbitrary star-shaped domains by smooth star-shaped domains,  it was shown in \cite{CWMonk2008} that the same $k$-dependence holds for arbitrary (not necessarily smooth) obstacles. Precisely, the following result was proved as \cite[Lemma 3.8]{CWMonk2008} in dimensions $n=2,3$ (cf.~\cite[Theorem 2.11]{Sp23}). The proof extends to all dimensions $n\geq 2$ once we have a version of  \cite[Lemma 2.1]{CWMonk2008} for dimensions $n>3$, which we provide as Lemma \ref{lem:2.1}.
\begin{theorem} \label{thm:ss}
If $\Gamma$ is star-shaped with respect to the origin, then (for $n\geq 2$)
$$
k^2C_{k,R} \leq k C_{k,R}' \leq n-1+2\sqrt{2} kR, \qquad k>0, \quad R>R_\Gamma. 
$$
\end{theorem}

The dependence on $k$ for large wavenumber implied by the above bound, that, for every $R>R_\Gamma$ and $k_0>0$, there exists $C_R>0$ such that
\begin{equation} \label{eq:nt}
C_{k,R} \leq C_{R} k^{-1}, \qquad k\geq k_0,
\end{equation}
holds more generally whenever $\Gamma$ is nontrapping\footnote{As usual (see \cite{chandler2020high} for more detail) we say that $\Gamma$ is nontrapping if, for some $R>R_\Gamma$, all billiard trajectories starting in $\Omega_R$ escape from $\Omega_R$ after some uniform time, and say that $\Gamma$ is trapping otherwise. The dashed lines in Figure \ref{fig:examples} are examples of trapped billiard trajectories, i.e., trapped rays.}. This classic result was first obtained by a combination of the results on propagation of singularities for the wave equation on manifolds with boundary by
Melrose and Sj\"ostrand \cite{MeSj:78, MeSj:82} with either the parametrix method of Vainberg \cite{Va:75} (see \cite{Ra:79} and \cite[\S4.6--4.7]{DyZw}) or the methods of Lax and Phillips \cite{LaPh:89} (see \cite{Me:79,SZ94}), following the proof
by Morawetz, Ralston, and Strauss  \cite{morawetz1975,MoRaSt:77} of the bound under a slightly stronger condition than nontrapping. We note that the bound \eqref{eq:nt} is optimal, in the sense that a simple quasimode construction (see, e.g., the discussion before Lemma 3.10 in \cite{CWMonk2008}) shows that, for every compact $\Gamma$, $k_0>0$, and $R>R_\Gamma$, there exists $C_R'>0$ such that 
\begin{equation} \label{eq:nt2}
C_{k,R} \geq C_R' k^{-1}, \qquad k\geq k_0.
\end{equation}

\begin{figure}[h]
\centering
\begin{tikzpicture}[line cap=round,line join=round,>=triangle 45,x=1.0cm,y=1.0cm, scale=1.4]
\colorlet{lightgray}{black!15}

\fill[color=lightgray](3.35,0.8) -- (4.25,0.8) -- (4.25,1.8) -- (3.35,1.8)--cycle;
\draw (3.35,0.8) -- (4.25,0.8) -- (4.25,1.8) -- (3.35,1.8)--cycle;
\fill[color=lightgray](4.75,0.8) -- (5.65,0.8) -- (5.65,1.8) -- (4.75,1.8)--cycle;
\draw (4.75,0.8) -- (5.65,0.8) -- (5.65,1.8) -- (4.75,1.8)--cycle;
\draw (3.02,1.8) node {(c)};
\draw[thick,dashed] (4.25,1.3) -- (4.75,1.3);

\fill[color=lightgray] (0.6,1.3) circle (0.5);
\draw (0.6,1.3) circle (0.5);
\fill[color=lightgray] (2,1.3) circle (0.5);
\draw (2,1.3) circle (0.5);
\draw (-0.2,1.8) node {(b)};
\draw[thick,dashed] (1.1,1.3) -- (1.5,1.3);

\coordinate (P) at ($(-1.75,1.3)+(130:0.25 and 0.5)$);
\coordinate (Q) at ($(-2.65,1.3)+(130:0.25 and 0.5)$);
\coordinate (R) at ($(-2.65,1.3)+(230:0.25 and 0.5)$);
\fill[color=lightgray] (P) arc (130:230:0.25 and 0.5) -- (R) -- (Q) -- (P);
\draw (P) arc (130:230:0.25 and 0.5) -- (R) -- (Q) -- (P);

\coordinate (A) at ($(-1.75,1.3)+(310:0.25 and 0.5)$);
\coordinate (B) at ($(-0.85,1.3)+(50:0.25 and 0.5)$);
\coordinate (C) at ($(-0.85,1.3)+(310:0.25 and 0.5)$);
\fill[color=lightgray] (A) arc (-50:50:0.25 and 0.5) -- (B) -- (C) -- (A);
\draw (A) arc (-50:50:0.25 and 0.5) -- (B) -- (C) -- (A);

\draw[red,dashed] (-1.75,1.3) ellipse (0.25 and 0.5);

\draw (-3.05,1.8) node {(a)};
\draw[thick,dashed] (-2,1.3) -- (-1.5,1.3);

\end{tikzpicture}
\caption{Examples of: (a) elliptic trapping; (b) hyperbolic trapping; (c) parabolic trapping. The dashed lines in each figure are trapped rays. The dashed red curve in (a) is the boundary of an ellipse.}
\label{fig:examples}
\end{figure}

In mild, so-called {\em hyperbolic trapping} (exemplified by the case where $\Gamma$ is the union of two disjoint convex obstacles with strictly positive curvature, Figure \ref{fig:examples}(b)) the growth of $C_{k,R}$ with $k$ is only logarithmically worse; the sharpest bound due to Burq \cite[Proposition 4.4]{Bu:04} is that, for every $R>R_\Gamma$,
\begin{equation}\label{eq:Ikawa2}
C_{k,R} \leq  C_R \frac{\log (2+k)}{k}, \qquad k\geq k_0,
\end{equation}
for some $C_R>0$. Similarly in so-called {\em parabolic trapping} (exemplified by the case where $\Gamma$ is the union of two squares or cubes with parallel sides, Figure \ref{fig:examples}(c)), the growth of $C_{k,R}$ with $k$ is only algebraically worse; the sharpest bound for the class of obstacles covered by \cite[Definition 1.4]{chandler2020high} is that, for every $R>R_\Gamma$,
\begin{equation}\label{eq:cw}
C_{k,R} \leq  C_R k, \qquad k\geq k_0,
\end{equation}
for some $C_R>0$; see \cite[Theorem 1.10]{chandler2020high}.

By contrast, if $\Gamma$ has an ellipse-shaped cavity (see Figure \ref{fig:examples}(a)) then exponential growth of $C_{k,R}$ can be achieved. Precisely, there exists an unbounded sequence of wavenumbers $0<k_1<k_2<\ldots$ and an $\alpha>0$ such that, for all $R>R_\Gamma$,
\begin{equation} \label{eq:ellipse}
C_{k_j,R}\geq C_R \re^{\alpha k_j}, \qquad j=1,2,\ldots,
\end{equation}
for some $C_R>0$;
see, e.g., \cite[\S2.5]{BetCha11}. More generally, if $\Gamma=\overline{\Omega_-}$, for some bounded, $C^\infty$ open set $\Omega_-$, and  there exists an elliptic trapped ray (i.e.~an elliptic closed broken geodesic),
and $\partial \Gamma=\partial \Omega_-$ is analytic in neighbourhoods of the vertices of the broken geodesic, then the resolvent can grow at least as fast as $\exp{(\alpha k_j^q)}$, through a sequence $k_j$ as above and for some range of $q\in(0,1)$, by the quasimode construction of Cardoso and Popov \cite{CaPo:02} (note that Popov proved \emph{superalgebraic} growth for certain elliptic trapped rays  in \cite{Po:91}).

If one restricts attention to smooth scatterers, i.e., $\Gamma=\overline{\Omega_-}$, for some bounded, $C^\infty$ open set $\Omega_-$, the exponential growth with $k$ exhibited in \eqref{eq:ellipse} is the worst possible. This follows from a
 general result of Burq \cite[Theorem 2]{Burq1998}, that, given any smooth $\Gamma$ and $k_0>0$, there exists $\alpha>0$ such that, for all $R>R_\Gamma$,
\begin{equation}\label{eq:Burq}
C_{k,R} \leq C_R \re^{\alpha k}, \qquad k\geq k_0,
\end{equation}
for some $C_R>0$. 

We have noted above that, for a smooth scatterer, exponential growth with $k$ of $C_{k,R}$ can be achieved through some sequence in cases of strong, elliptic trapping. But it has been shown recently by Lafontaine et al.~\cite{lafontaine2021most} that, for most wavenumbers, i.e., for all $k\geq k_0$ outside some set $J$ of arbitrarily small measure, growth is at worst algebraic in the case that $\Gamma = \overline{\Omega_-}$ for some Lipschitz domain $\Omega_-$.

\begin{theorem}[{\cite[Theorem 1.1]{lafontaine2021most}}] \label{thm:most}
Let $\Omega_-$ be a bounded Lipschitz open set such that $\Omega:= \R^n\setminus \Gamma$ is connected, where  $\Gamma:=\overline{\Omega_-}$. Then, given $k_0>0$, $R>R_\Gamma$, $\delta>0$,  and $\varepsilon>0$, there exists $C=C(k_0,R, \delta,\varepsilon,n)$ and a set $J\subset [k_0,\infty)$ with Lebesgue measure $|J|\leq \delta$, such that
$$
C_{k,R} \leq C k^{5n/2+\varepsilon}, \qquad k \in [k_0,\infty)\setminus J.
$$
\end{theorem}

\subsection{Our main results} \label{sec:main}
The main contribution of this paper is to show that, in contrast to the bounds \eqref{eq:nt},  \eqref{eq:Ikawa2}, \eqref{eq:cw}, and \eqref{eq:Burq} above, if one does not impose any smoothness or other geometrical constraint on $\Gamma$, then, in a sense we make precise, there is no constraint on growth, through some sequence, of the cut-off resolvent as a function of $k$. 
The following is our main result.
\begin{theorem} \label{thm0} Suppose the sequences $(k_j)_{j\in \mathbb{N}}$ and $(a_j)_{j\in \mathbb{N}}$ are positive, increasing and unbounded, and that 
\begin{equation} \label{eq:sumfin}
\sum_{j=1}^\infty k_j^{-n} < \infty.
\end{equation}
Then there exists a compact $\Gamma\subset \R^n$, with $\Omega := \R^n\setminus \Gamma$ connected, such that, for all $R>R_\Gamma=\max_{x\in \Gamma} |x|$ it holds that $C_{k_j,R}=\lVert \chi_R R(k_j)\chi_R\rVert_{L^2(\Omega)\rightarrow L^2(\Omega)}\geq a_j$ for each $j\in \mathbb{N}$.
\end{theorem}

\begin{figure}[t]
\[
\includegraphics[width=120mm]{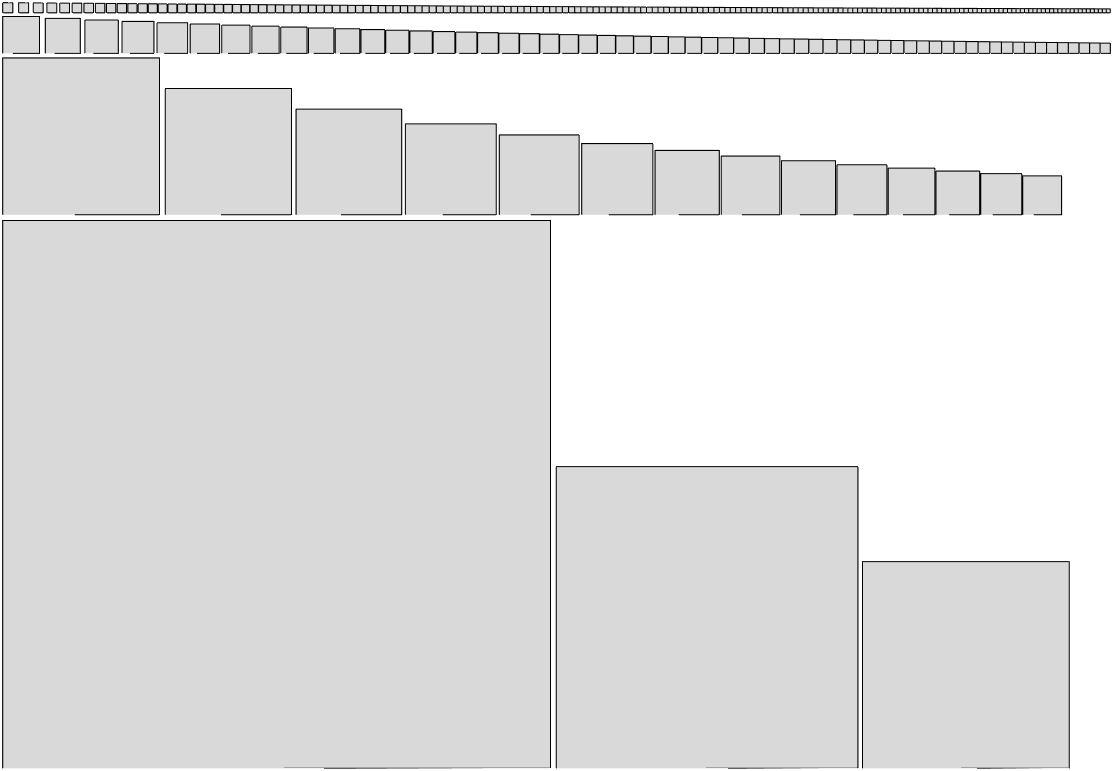} 
\]
\vspace*{-3ex}
\caption{Plot, in the case $n=2$,  of the first four of the infinitely many layers of a compact set $\Gamma\subset [0,3]^2$ constructed as in Theorem \ref{thm1} and the proof of Theorem \ref{thm2} ($\Gamma$ is the closure of the union of the black solid lines).  $\Gamma$ is a subset of $\partial \mO_\infty$, where the bounded open set $\mO_\infty$ (shaded in grey) is a countable union of disjoint squares (the first 255
squares are shown in this image). For further details see below Theorem \ref{thm2}.} \label{fig:Gamma}
\end{figure}

Theorem \ref{thm0} is a corollary of our other main result, Theorem \ref{thm1} below, that, coupled with Theorem \ref{thm2} (a variant on the main theorem of \cite{MM68}) makes an explicit construction of the $\Gamma$ whose existence is asserted in Theorem \ref{thm0}. (Combining Theorem \ref{thm1} instead with Theorem \ref{thm2v} provides a proof of Theorem \ref{thm0} with \eqref{eq:sumfin} replaced by the marginally weaker condition \eqref{eq:weaker}.) The choice of the translations $t_j\in \R^n$, referenced in Theorem \ref{thm2}, is made explicit in the associated Algorithm \ref{alg:trans} (and the corresponding choice in Theorem \ref{thm2v} is made explicit in the proof of that theorem).

These theorems construct $\Gamma$ as a subset of the boundary of  a bounded open set $\mO_\infty$ that is the union of a sequence $\mO_1,\mO_2,\ldots$ of cubes whose closures are disjoint (for examples in the case $n=2$, discussed in more detail below Theorem \ref{thm2}, in Example \ref{ex:vert}, and below Theorem \ref{thm2v}, respectively, see Figures \ref{fig:Gamma}--\ref{fig:GammaE}). The $j$th cube $\mO_j$ has sidelength $\ell_j$ chosen so that $k_j^2$ is the first Dirichlet eigenvalue of $-\Delta$ in $\mO_j$. It follows that $k_j^2$ is also a Dirichlet eigenvalue of $-\Delta$ in $\mO_\infty$. Further, the corresponding eigenfunction, extended by zero from $\mO_j$ to $\Omega = \R^n\setminus \Gamma$, is, with the choice of $\varepsilon_j$ that we make in Theorem \ref{thm1},  a sufficiently good quasimode of the standard weak formulation  of the scattering problem that it follows that $C_{k_j,R}\geq a_j$ for every $R>R_\Gamma$ (for details see the proof of Theorem \ref{thm1}).


In Theorems \ref{thm1}, \ref{thm2}, and \ref{thm2v} below we use the following notations. For $S\subset \R^n$, $\ell>0$, and $t\in \R^n$,  $\ell S + t:= \{\ell x+t:x\in S\}$ is $S$ scaled by a factor $\ell$, then translated by the vector $t$. Define
$\mathcal{O}:=\{x=(x_1,...,x_n)\in \mathbb{R}^n: \text{$0<x_i<1$ for $i=1,...,n$}\}$, so that $\mO$ is an open unit cube. Define the set $\Gamma^\varepsilon\subset \partial \mathcal{O}$, for $0<\varepsilon<1$, by
\begin{equation*}
\Gamma^\varepsilon=\partial \mathcal{O}\setminus\{x=(x_1,...,x_n)\in\R^n: \text{$x_n=0$ and $0<x_i<\varepsilon$ for $i=1,...,n-1$}\},
\end{equation*}
so that $\Gamma^\varepsilon$ is $\partial \mO$ with an $(n-1)$-dimensional cube of sidelength $\varepsilon\in (0,1)$ removed from one of its faces, and $\R^n\setminus \Gamma^\varepsilon$ is connected. If $\varepsilon \ll 1$ then, in the context of scattering problems, $\Gamma^\varepsilon$ is a {\em Helmholtz resonator} in the sense, e.g., of \cite{Gady,Gady2}, or an  {\em open cavity}, in the sense, e.g., of  \cite{Bruno}.  
\begin{theorem}\label{thm1}
Suppose the sequences $(k_j)_{j\in \mathbb{N}}$ and $(a_j)_{j\in \mathbb{N}}$ are positive, increasing and unbounded. For $j\in \N$ let $\ell_j:=\pi \sqrt{n}k_j^{-1}$, $t_j\in \R^n$, and $\mO_j := \ell_j\mO+t_j$, so that $\mO_j$ is a cube of sidelength $\ell_j$,  and let $\Gamma_j:= \ell_j\Gamma^{\varepsilon_j}+t_j \subset \partial \mathcal{O}_j$, where
\begin{equation}\label{gaplength}
\varepsilon_j:=\left(\frac{3}{2\pi^2}\right)^{1/3}\, \bigg(1+2k_j\sqrt{2k_j^2 a_j^2 +a_j}\,\bigg)^{-2/(3n-3)} \, \in \, (0,1),
\end{equation}
so that $\Gamma_j$ is $\partial \mO_j$ with an $(n-1)$-dimensional cube of sidelength $\ell_j\varepsilon_j\in (0,\ell_j)$ removed from one of its faces.
Furthermore, suppose that $(k_j)_{j\in \N}$ and the translations $t_j\in \R^n$ are such that: 
\begin{itemize}
\item[(i)] $\overline{\mathcal{O}_i}\cap \overline{\mathcal{O}_j}=\varnothing$ for $i\neq j$, and hence $\Gamma_i \cap \Gamma_j=\varnothing$ for $i\neq j$;
\item[(ii)] $\mO_\infty := \bigcup_{j=1}^\infty \mathcal{O}_j$ is bounded, so that $\Gamma:=\overline{\bigcup_{j=1}^\infty \Gamma_j}\subset \partial \mO_\infty$ is compact;
\item[(iii)] $\Omega:= \mathbb{R}^n\setminus \Gamma$ is connected.
\end{itemize}
Then, for every $R>R_\Gamma$, it holds that $C_{k_j,R}=\lVert \chi_R R(k_j)\chi_R\rVert_{L^2(\Omega)\rightarrow L^2(\Omega)}> a_j$ for each $j\in \mathbb{N}$.
\end{theorem}

In the above theorem, assuming  that $\varepsilon_j \ll 1$ (which holds if $k_j^2 a_j \gg 1$), $\Gamma_j$ is a Helmholtz resonator for each $j$, tuned to have a complex resonance close to $k_j$ (see, e.g., \cite[\S2.5]{Gady}). Thus, assuming that (i)-(iii) hold, the theorem constructs $\Gamma$ as the closure of the union of countably many disjoint Helmholtz resonators, constructed so that each $k_j$ is almost a resonance of the scattering problem.

The  following example (see Figure \ref{fig:GammaV}) illustrates the above result, providing a simple arrangement of the $\mO_j$ and $\Gamma_j$ such that (i)-(iii) hold. Note that this construction requires that $k_j$ grow as $j\to \infty$ at a significantly faster rate than \eqref{eq:sumfin}.

\begin{figure}[t]
\[
\includegraphics[width=90mm]{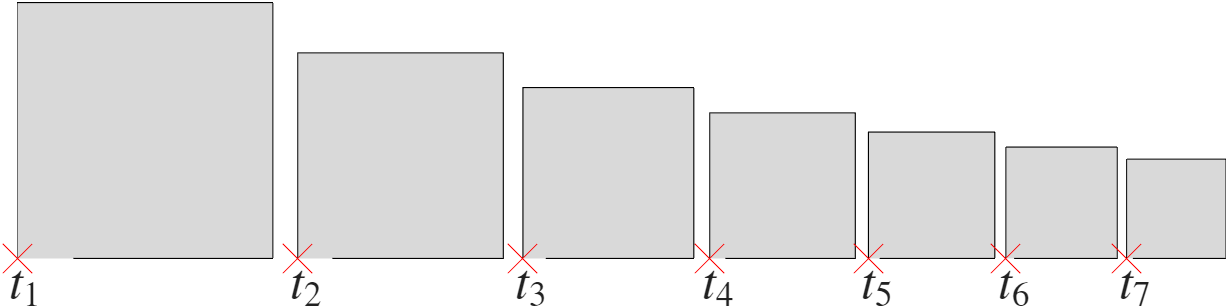} 
\]
\vspace*{-3ex}
\caption{Plot, in the case $n=2$,  of (part of) the compact set $\Gamma\subset \R^2$ constructed as in Example \ref{ex:vert} ($\Gamma$ is the closure of the union of the black solid lines).  $\Gamma$ is a subset of $\partial \mO_\infty$, where the bounded open set $\mO_\infty$ (shaded in grey) is a countable union of disjoint squares (the first seven squares are shown in this image). The red crosses mark the locations of the points $t_1,t_2,\ldots$. In this example $k_j=4(j+4)^{6/5}$, $d_j = (j+7)^{-2}$, and $a_j = j^{1/4}/1000$, for $j\in \N$. For further discussion see Remark \ref{rem:layers}.} \label{fig:GammaV}
\end{figure}

\begin{example} \label{ex:vert} {\em (An application of Theorem \ref{thm1})}
Suppose that 
\begin{equation}\label{strongcond}
\sum_{j=1}^\infty k_j^{-1}<\infty
\end{equation}
and $\ell_j$ is defined as in Theorem \ref{thm1}, so that
$$
\sum_{j=1}^\infty \ell_j < \infty.
$$
Define the translations $(t_j)\subset \R^n$ by $t_1=0$ and 
$$
t_{j+1} = t_j + (\ell_{j} + d_j)e_1, \qquad j \in \N,
$$
where $e_1$ is the unit vector in the first coordinate direction and $(d_j)_{j\in \N}$ is any positive, summable sequence, and, with this choice of $(t_j)$, define $\mO_j$ and $\Gamma_j$, for $j\in \N$, as in Theorem \ref{thm1}, noting that the definition of $\Gamma_j$ depends, through \eqref{gaplength}, on the unbounded sequence $(a_j)_{j\in \N}$. This  choice of $(t_j)_{j\in \N}$ places the cubes $(\mO_j)_{j\in \N}$ in a horizontal row, with $\mO_{j+1}$ to the right of $\mO_j$ and $\dist(\mO_{j+1},\mO_j)=d_j$, $j\in \N$. Conditions (i)-(iii) of Theorem \ref{thm1} are satisfied, and
\begin{equation} \label{eq:G}
\Gamma =\overline{\bigcup_{j=1}^\infty \Gamma_j} = \{s e_1\}\cup \bigcup_{j=1}^\infty \Gamma_j, \quad \mbox{where} \quad s:= \sum_{j=1}^\infty(\ell_{j}+d_j).
\end{equation}
Thus Theorem \ref{thm1} implies that, if the positive increasing sequence $(k_j)_{j\in \N}$ satisfies \eqref{strongcond} and $\Gamma$ is given by \eqref{eq:G}, then $C_{k_j,R} > a_j$ for $j\in \N$ and $R>R_\Gamma$. 
\end{example}

The above example shows that Theorem \ref{thm1} can be applied to bound $C_{k_j,R}$ if $(k_j)_{j\in \N}$ is positive and increasing and satisfies \eqref{strongcond}. These conditions imply that $k_j$ increases rather rapidly, in particular that, for some $C>0$, $k_j\geq Cj$, for $j\in \N$ (cf.~\eqref{eq:bb} below). In the following remark we point out that \eqref{eq:sumfin} is a necessary condition  for conditions (i) and (ii) in Theorem \ref{thm1} to hold for some choice of the translations $(t_j)_{j\in \N}$, and note other constraints on $k_j$ that (i) and (ii) imply. Theorem \ref{thm2} below will show that \eqref{eq:sumfin} is also a sufficient condition for (i)-(iii) in Theorem \ref{thm1} to hold.

\begin{remark} \label{rem:const} {\em (The growth of $(k_j)_{j\in \N}$ in Theorem \ref{thm1}.)}
Suppose (i) in Theorem \ref{thm1} holds. Then, with $\ell_j$ as defined there and with $|\mO_\infty|_n$ denoting the $n$-dimensional Lebesgue measure of $\mO_\infty$,
\begin{equation} \label{eq:vol}
|\mO_\infty|_n = \sum_{j=1}^\infty \ell_j^n = \pi^n n^{n/2}  \sum_{j=1}^\infty k_j^{-n}.
\end{equation}
Thus, if (ii) in Theorem \ref{thm1} also holds, so that $\mO_\infty$ is bounded, \eqref{eq:sumfin} follows.

Since we assume in Theorem \ref{thm1} that $(k_j)_{j\in \N}$ is increasing, it follows from \eqref{eq:sumfin}  (e.g., \cite[\S179]{Hardy}) that
\begin{equation} \label{eq:abel}
jk_j^{-n} \to 0 \quad \mbox{as} \quad  j\to\infty.
\end{equation}
Further, quantitatively, it follows from \eqref{eq:vol} and  \eqref{eq:sumfin}  that
\begin{equation} \label{eq:bb}
j k_j^{-n} \leq \sum_{i=1}^j k_i^{-n} \leq  \pi^{-n} n^{-n/2} |\mO_\infty|_n, \qquad j\in \N,
\end{equation}
so that 
\begin{equation} \label{eq:2ndbound}
j^{-1/n} k_j \geq  \pi n^{1/2} \left(|\mO_\infty|_n\right)^{-1/n}, \qquad j\in \N.
\end{equation}

A further lower bound on $k_j$  follows from the observation that, if (i) and (ii) in Theorem \ref{thm1} hold, then, for each $j\in \N$, $k_j^2$ is an eigenvalue of $-\Delta$ in $\mO_j$, and so is also a Dirichlet eigenvalue for the bounded open set $\mO_\infty$. Thus, where $\lambda_j$ denotes the $j$th such eigenvalue, with  the eigenvalues listed in non-decreasing order and counted according to multiplicity so that $0<\lambda_1 \leq \lambda_2 \ldots$, it holds that $k_j\geq \lambda^{1/2}_j$, for $j\in \N$. Thus, and by the Weyl law (see \cite[Equation (3.3.4)]{levitin2023topics}),
$$
\liminf_{j\to\infty} j^{-1/n} k_j \geq \lim_{j\to\infty} j^{-1/n}\lambda^{1/2}_j = \left(\frac{|\mO_\infty|_n}{(4\pi)^{n/2} \Gamma(1+n/2)}\right)^{-1/n}.
$$
But this is a weaker constraint than \eqref{eq:abel} which implies that $j^{-1/n} k_j\to \infty$ as $j\to\infty$. It is also weaker than the quantitative constraint  \eqref{eq:2ndbound} since $2(\Gamma(1+n/2))^{1/n} \leq \sqrt{\pi n}$, for $n\in \N$. (This follows by induction on $n$ for odd and even $n$ separately.)
\end{remark}

The  following remark explores the relationship between (i)-(iii) in the above theorem. Here and subsequently we use the notation 
$$
B_R:= \{x\in \R^n:|x|<R\}, \qquad R>0.
$$ 

\begin{remark} {\em (Conditions (i) and (ii) in Theorem \ref{thm1} do not imply condition (iii))} \label{rem:i23} It might be imagined that (i)-(ii) in Theorem \ref{thm1} automatically imply that (iii) holds. This is not the case, as is shown by the following example, inspired by the ``Swiss cheese sets'' of Polking \cite{Polking72} (and see  \cite[Example 8.3]{chandler2018well}). 

Let $\{q_j\}_{j\in \mathbb{N}}$ denote a dense set of distinct points in $\partial B_1\setminus \{\pm e_1,\ldots,\pm e_n\}$, where $e_m$ denotes the unit vector in the $m$th coordinate direction. Let us describe a recursive placement of the cubes $\mO_1,\mO_2,\ldots$ in the open unit ball $B_1$, such that all the conditions of Theorem \ref{thm1} except (iii) are satisfied. Set $\mO_1 := \ell_1\mO + t_1$, choosing $\ell_1 > 0$ and $t_1\in \R^n$ so that $\mO_1 \subset \overline{B_1}$ and $B_{1/2}\cap \mO_1=\emptyset$, and such that $\partial \mO_1 \cap \partial B_1=\{q_1\}$; this last requirement is possible since $q_1\neq \pm e_m$, for $m=1,\ldots,n$. Having chosen $\mO_1,\ldots,\mO_j$, for some $j\in \N$, set $\mO_{j+1} := \ell_{j+1} \mO+t_{j+1}$, choosing $\ell_{j+1}>0$ and $t_{j+1}\in \R^n$ so that $\ell_{j+1}<\ell_j$, $\mO_{j+1} \subset \overline{B_1}$, $B_{1/2}\cap \mO_{j+1}=\emptyset$, and $\partial \mO_{j+1} \cap \partial B_1=\{q_{j+1}\}$, and such that $\overline{\mO_{j+1}}\cap \overline{\mO_i}=\emptyset$, $i=1,\ldots,j$. 

The sequence $(\mO_j)_{j\in \N}$ of cubes formed by this construction clearly satisfies (i) and (ii)  in Theorem \ref{thm1}. In particular, $\mO_\infty:= \bigcup_{j=1}^\infty \mO_j \subset B_1$ and $B_{1/2}\cap \mO_\infty = \emptyset$. Further, since  $\partial \mO_j \cap \partial B_1=\{q_j\}$, for $j\in \N$, it follows that $\Gamma_j\cap \partial B_1=\{q_j\}$, for $j\in \N$, where  $\Gamma_j := \ell_j \Gamma^{\varepsilon_j} + t_j$, for some $\varepsilon_j\in (0,1)$ (note that this definition for $\Gamma_j$ implies that $\Gamma_j\subset \partial \mO_j$, and that $\Gamma_j$  contains all the corners of the cube $\mO_j$). It follows that $\partial B_1 \subset \Gamma := \overline{\bigcup_{j=1}^\infty \Gamma_j} \subset \overline{B_1}$ and that $B_{1/2}\cap \Gamma=\emptyset$. Thus (iii) is not satisfied, i.e., $\Omega := \R^n\setminus \Gamma$ is not connected. Further, if $(k_j)_{j\in \N}$ is defined by $k_j:= \pi \sqrt{n}\ell_j^{-1}$, $j\in \N$,  then $k_j$ and $\ell_j$ are related as in Theorem \ref{thm1} and $(k_j)_{j\in \N}$ is a positive, increasing, unbounded sequence. If $(a_j)_{j\in \N}$ is some other positive, increasing, unbounded sequence and $\varepsilon_j$ is defined by \eqref{gaplength}, for $j\in \N$, then all the conditions of Theorem \ref{thm1} are satisfied, except that (iii) does not hold. 
\end{remark}

Our next theorem shows that, if $(k_j)_{j\in\N}$  satisfies \eqref{eq:sumfin} and is increasing, then it is possible to choose the translations $(t_j)_{j\in \N}$ so that (i)-(iii) in Theorem \ref{thm1} hold, so that $C_{k_j,R}>a_j$ for each $j\in \N$, establishing Theorem \ref{thm0}. This result is, with some additions to justify that (iii) in Theorem \ref{thm1} holds and to provide explicit formulae for the translations $(t_j)_{j\in\N}$, a result on the packing of cubes inside a larger cube 
 that is the main result of \cite{MM68}, and that we reproduce as Theorem \ref{thm:cube} below. 
 

\begin{theorem}\label{thm2}
Suppose the sequence $(\varepsilon_j)_{j\in \N}\subset (0,1)$, that $(\ell_j)_{j\in \N}\subset (0,\infty)$ is decreasing,
\begin{equation} \label{eq:ellb}
V := \sum_{j=1}^\infty \ell_j^n < \infty, \quad \mbox{and} \quad a>\ell_1+(V-\ell_1^n)^{1/n}.
\end{equation}
Then there exists a choice of the sequence of translations $(t_j)_{j\in \N}\subset \R^n$ such that the sequences  $(\mO_j)_{j\in \N}$ and $(\Gamma_j)_{j\in \N}$, defined by
$$
\mO_j := \ell_j \mO+t_j \quad \mbox{and} \quad \Gamma_j := \ell_j \Gamma^{\varepsilon_j} + t_j, \qquad j\in \N,
$$
satisfy (i)-(iii) of Theorem \ref{thm1},  with $\mO_\infty = \bigcup_{j=1}^\infty \mO_j\subset a\mO$.
\end{theorem}

A specification for the translations $(t_j)_{j\in \N}\subset \R^n$ that achieve (i)-(iii) is implicit in \cite[Theorem 1]{MM68}, and is made explicit  in \S\ref{sec:main2} in the proof of Theorem \ref{thm2} and  in Algorithm \ref{alg:trans}. 
Briefly, the sequence of cubes $(\mO_j)_{j\in \N}$ is arranged in layers, at different heights in the $n$th coordinate direction,  with a spacing $d_i>0$ between layer $i$ and layer $i+1$, for some sequence $(d_i)_{i\in \N}\in \ell^1(\N)$. The number of cubes in each layer, and the number of layers, in particular whether finitely many or infinitely many layers are needed, is determined by Algorithm \ref{alg:trans} as a function of the sequence $(\ell_j)_{j\in \N}$.

In Figure \ref{fig:Gamma}, in the case $n=2$ with $k_j := 2j^{1/2} \log^{3/2}(j+\re)$, for $j\in \N$, which satisfies \eqref{eq:sumfin}, we show part of this construction (the first four of the infinitely many layers that are needed in this case) if Algorithm \ref{alg:trans} has input $a=3$ and the choice $d_i := (i+7)^{-2}$, $i\in \N$, is made on line 3 of the algorithm. (With these choices it follows that $\ell_1\approx1.476$ and
$V:=|\mO_\infty|_n\approx4.1024$, so that 
$a-\ell_1-(V-\ell_1^2)^{1/2}\approx0.1370>
\epsilon=\sum_{i=1}^\infty d_i\approx0.1331$.) In this example, so that the values of $\varepsilon_j$ given by \eqref{gaplength} lead to visible gaps (i.e., $\Gamma_j$ is distinguishable from $\partial \mO_j$), we choose a very slow growth and very small values for $a_j$, defining $a_j := j^{1/4}/10,000$, $j\in \N$. We emphasise that changing the sequence $(a_j)_{j\in \N}$ does not affect the size or placement of the boxes $\mO_j$, just the size of the gaps $\varepsilon_j$.

In \S\ref{sec:main2} we also prove a variant of Theorem \ref{thm2} (Theorem \ref{thm2v} below), included because it proves almost as strong a result as Theorem \ref{thm2} but with a simpler and  markedly different arrangement of the cubes $(\mO_j)_{j\in \N}$.  The constraint on $(\ell_j)_{j\in \N}$ in this variant, that \eqref{eq:ellbv} holds for some $c,\eta>0$ and $M\in \N$, implies \eqref{eq:ellb}. Thus, except that we have additionally \eqref{eq:onep}, this is a weaker result than Theorem \ref{thm2}, but only ``logarithmically weaker'' in that the constraint \eqref{eq:ellbv} is satisfied if $(\ell_j\log^2(j))_{j\in \N}$ is decreasing at infinity and satisfies a logarithmically weaker version of \eqref{eq:ellb}, that
$$
\sum_{j=1}^\infty \ell_j^n \log^2(j) < \infty.
$$
(This claim holds since these assumptions on $\ell_j$ imply (cf.~\eqref{eq:abel}) that $j\ell_j^n\log^2(j)=o(1)$ as $j\to \infty$, so that $\ell_j = o\left((j\log(j))^{1/n}\log(\log(j))\right)^{-1}$, so that \eqref{eq:ellbv} holds with $M=2$ and $\eta=1$.) 

In a recent preprint  (see \cite[\S4]{ChFr25}), Chaumont-Frelet has proved in the case $n=3$ a result than is less sharp than Theorem \ref{thm2v} below  (and so  less sharp than Theorem \ref{thm2}), showing, by a similar construction to that of \cite[Theorem 1]{MM68}, that if, for some $c>0$ and $\varepsilon>0$, $0<\ell_j\leq c/ j^{1/3+\varepsilon}$, for $j\in \N$, then there exists  $(t_j)_{j\in \N}\subset \R^3$ such that (i) and (ii) in Theorem \ref{thm1} hold.

In the following theorem we use, for $m\in \N$ and $x>0$, the notations
$$
\re_1 := \re, \quad \re_{m+1} := \exp(\re_m), \quad \log_1(x) := \log(x), \quad \log_{m+1}(x) := \log(\log_m(x)),
$$  
adopting in \eqref{eq:ellbv} the usual convention that the product evaluates as $1$ if $M=1$. Combining Theorems \ref{thm1} and \ref{thm2v} leads to a proof of Theorem \ref{thm0} with  \eqref{eq:sumfin} replaced by the (marginally) weaker condition that, for some $c,\eta>0$ and $M\in \N$,
\begin{equation} \label{eq:weaker}
k_j \geq c \left(j \prod_{m=1}^{M-1}\log_m(j+\re_m)\right)^{1/n}\log_{M}^{1+\eta}(j+\re_{M}), \quad j\in \N.
\end{equation}
The relationship between \eqref{eq:sumfin} and \eqref{eq:weaker} is that, for every $c,\eta>0$ and $M\in \N$, \eqref{eq:weaker} implies \eqref{eq:sumfin}, but this is not true, for any $M\in \N$, if we delete the $\log_M^{1+\eta}(j+\re_{M})$ term in \eqref{eq:weaker}. On the other hand, as we noted in Remark \ref{rem:const}, the assumptions of Theorem \ref{thm0} imply that $j^{-1/n}k_j\to\infty$ as $j\to\infty$, a weakened version of \eqref{eq:weaker}.
\begin{theorem}\label{thm2v}
Suppose the sequence $(\varepsilon_j)_{j\in \N}\subset (0,1)$ and that $(\ell_j)_{j\in \N}\subset (0,\infty)$ is decreasing, $c',\eta>0$, $M\in \N$, and
\begin{equation} \label{eq:ellbv}
\ell_j \leq \frac{c'}{ \left(j \prod_{m=1}^{M-1}\log_m(j+\re_m)\right)^{1/n}\log_{M}^{1+\eta}(j+\re_{M})}, \quad j\in \N.
\end{equation}
Then there exists a choice of the sequence of translations $(t_j)_{j\in \N}\subset \R^n$ such that the sequences  $(\mO_j)_{j\in \N}$ and $(\Gamma_j)_{j\in \N}$, defined as in Theorem \ref{thm2},
satisfy (i)-(iii) of Theorem \ref{thm1}. Further, for some $x_\infty\in \R^n$,
\begin{equation} \label{eq:onep}
\Gamma = \overline{\bigcup_{j=1}^\infty \Gamma_j} = \{x_\infty\} \cup \bigcup_{j=1}^\infty \Gamma_j.
\end{equation}
\end{theorem}

\begin{figure}[t]
\[
\includegraphics[width=20mm]{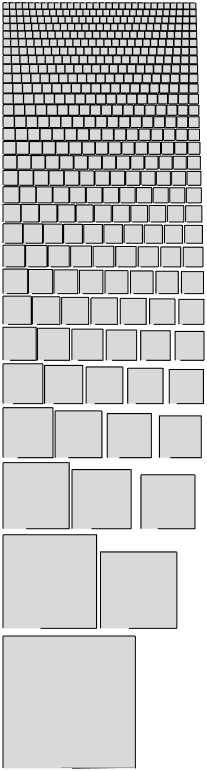} 
\]
\vspace*{-3ex}
\caption{Plot, in the case $n=2$,  of the first 30 layers of the compact set $\Gamma\subset \R^2$ constructed as in Theorem \ref{thm1} and the proof of Theorem \ref{thm2v} ($\Gamma$ is the closure of the union of the black solid lines).  $\Gamma$ is a subset of $\partial \mO_\infty$, where the bounded open set $\mO_\infty$ (shaded in grey) is a countable union of disjoint squares (the first 465 
squares are shown in this image). Note that in this construction there are precisely $N_i=i$ squares in level $i$ and that the construction in this figure and in Figure \ref{fig:Gamma} differ only in the choice of the translation vectors $(t_j)_{j\in \N}$. For further details see below Theorem \ref{thm2v}.} \label{fig:GammaE}
\end{figure}

The construction in the proof of the above theorem has the feature \eqref{eq:onep}, that $\Gamma$ differs from $\bigcup_{j=1}^\infty \Gamma_j$ at the single point $x_\infty$, which leads, as we discuss in \S\ref{sec:fq}, to a variant of this construction where the boundary of $\Gamma$ is Lipschitz continuous except at a single point.  The boxes $(\mO_j)_{j\in \N}$ are arranged in layers in the proof of Theorem \ref{thm2v}, as in the proof of Theorem \ref{thm2}, with a spacing $d_i$ between layer $i$ and $i+1$. But a significant difference, leading to \eqref{eq:onep}, is that we make the {\em a priori} specification that there are $N_i=\lfloor i\prod_{m=1}^{M-1}\log_m(i+\re_m)\rfloor^{n-1}$ boxes in the $i$th layer, where $\lfloor\cdot\rfloor$ is the floor function. In Figure \ref{fig:GammaE} we plot, for the same choices of $(k_j)_{j\in \N}$ and $(d_j)_{j\in \N}$ as in Figure \ref{fig:Gamma} (so \eqref{eq:ellbv} holds with $c=2$, $\eta=1/2$, and $M=1$, and $N_i=i$, for $i\in \N$), the first 30 of the infinitely many layers in the construction of $\Gamma$ in the proof of Theorem \ref{thm2v}.

We remark that we have recently applied the results and arguments of Theorems \ref{thm1} and \ref{thm2} to establish lower bounds on the worst case growth of the condition number in first kind integral equation formulations of the scattering problem of \S\ref{sec:intro} for a general compact scatterer $\Gamma$; see \cite[Proposition 2.7]{SiavashSimon0}.

\subsection{Further questions} 
\label{sec:fq} We note two further questions related to our main results.  The first of these we address elsewhere; the second appears open.

\paragraph{Question 1} Theorem \ref{thm0} tells us that $C_{k,R}$ can increase arbitrarily fast through any mildly increasing sequence $(k_j)_{j\in \N}$ that satisfies \eqref{eq:sumfin} if our only requirement on $\Gamma$ is that it is compact  with connected complement. So  does Theorem \ref{thm:most} still hold if its requirement that $\Gamma$ be the closure of a bounded Lipschitz domain is relaxed to allow arbitrary compact $\Gamma$? In a paper in preparation \cite{SiavashSimon1} we answer this question in the affirmative, at the same time making a slight  sharpening of the original proof, showing the following result.

\begin{theorem} \label{thm:most2}
Suppose that $\Gamma\subset \R^n$ is compact and  $\Omega:= \R^n\setminus \Gamma$ is connected. Then, given $k_0>0$, $R>R_\Gamma$, $\delta>0$,  and $\varepsilon>0$, there exists $C=C(k_0,R, \delta,\varepsilon,n)$ and a set $J\subset [k_0,\infty)$ with Lebesgue measure $|J|\leq \delta$, such that
$$
C_{k,R} \leq C k^{2n+\varepsilon}, \qquad k \in [k_0,\infty)\setminus J.
$$
\end{theorem}
Our proof of this result uses the general {\em black-box} version,  \cite[Theorem 3.3]{lafontaine2021most}, of Theorem \ref{thm:most} from \cite{lafontaine2021most}, refining the proof of that result by use of  a sharper semiclassical maximum principle.

\paragraph{Question 2} Theorem \ref{thm0} tells us that $C_{k,R}$ can grow arbitrarily fast, through some sequence, as a function of $k$,  if all we require is that $\Gamma$ is compact with connected complement, whereas no worse than exponential growth is possible, by the bound \eqref{eq:Burq} of \cite{Burq1998}, if $\Gamma=\overline{\Omega_-}$, for some bounded $C^\infty$ domain $\Omega_-$. So, if $\Gamma=\overline{\Omega_-}$ and $\Omega_-$ is not $C^\infty$ but has less smoothness, e.g., is Lipschitz or $C^0$ (as defined, e.g., in \cite{mclean2000strongly}), is there still a limit on how fast $C_{k,R}$ can grow as $k\to \infty$, and is this limit still one of exponential growth? 

Related to this question one can modify  our construction of $\Gamma$ in Theorems \ref{thm1} and \ref{thm2v} (which together imply Theorem \ref{thm0} with \eqref{eq:sumfin} replaced by the marginally weaker \eqref{eq:weaker}), replacing $\mO$ and $\Gamma^\varepsilon$  by thickened versions, $\widetilde \mO := (-1,2)^n$ and
\begin{eqnarray*}
\widetilde \Gamma^{\varepsilon} &:=& \overline{\widetilde \mO} \setminus\\
& & (\mO \cup \{(x_1,...,x_n)\in \R^n: \text{$-1 \leq x_n\leq 0$ and $0<x_i<\varepsilon$ for $i=1,...,n-1$}\}),
\end{eqnarray*}
and defining each $\Omega_j$ and $\Gamma_j$ with $\mO$ replaced by $\widetilde \mO$ and $\Gamma^\varepsilon$ replaced by $\widetilde \Gamma^{\varepsilon}$. Theorem \ref{thm1} holds, as written, with this modification, and the construction of Theorem \ref{thm2v} follows through with trivial edits. Since the complement of $\widetilde \Gamma^\varepsilon$ is a connected Lipschitz domain, the boundary of the resultant connected domain $\Omega:= \R^n\setminus \Gamma$ is locally the graph of a Lipschitz function at all $x$ except the single point $x\in \Gamma\cap \partial \Omega$ that is not in $\Gamma_j$ for some $j\in \N$ (see \eqref{eq:onep}). This demonstrates that it is enough for $\partial \Omega$ to fail to be locally a Lipschitz graph at a single point for arbitrarily fast growth of $C_{k,R}$, through some sequence, as a function of $k$, to be achievable.

\subsection{Outline of the paper} Let us briefly describe the remainder of the paper. 
Section \ref{sec:fs} recalls function space notations and definitions. Section \ref{sec:weak} recalls a standard variational formulation of the scattering problem and its associated sesquilinear form, and relates the inf-sup constant of that form to the norms $C_{k,R}$ and $C'_{k,R}$ of the cut-off resolvent that we have introduced in \eqref{eq:cutoff} and \eqref{eq:Cdashdef}. The results in \S\ref{sec:weak} are known in dimensions $n=2,3$, but the extension to $n>3$ appears to be new in part. In \S\ref{sec:mult} we state and prove, with explicit constants, a special case of the multiplicative trace inequality.  
In \S\ref{sec:main2} the results of \S\ref{sec:prelim} are applied to prove our main results, Theorems \ref{thm1}, \ref{thm2}, and \ref{thm2v}. (Recall that Theorem \ref{thm0} is an immediate corollary of Theorems \ref{thm1} and \ref{thm2}.) In the appendix we prove properties of radiating solutions of the Helmholtz equation that are well known in dimensions $n=2,3$, but are not yet fully established for dimensions $n>3$.

\section{Preliminaries} \label{sec:prelim}

\subsection{Function spaces} \label{sec:fs}
For definiteness, we note our (largely standard) function space notations. Given a domain (i.e., a non-empty open set) $\Omega\subset\mathbb{R}^n$, 
$$
H^1(\Omega):= \{u\in L^2(\Omega):\nabla u\in (L^2(\Omega))^n\},
$$
where $\nabla u$ denotes the weak or distributional gradient.  We recall that $H^1(\Omega)$ is a Hilbert space equipped with the norm $\|\cdot\|_{H^1(\Omega)}$ defined by
\begin{equation*}
\|v\|^2_{H^1(\Omega)} = \|v\|_{L^2(\Omega)}^2 + \| |\nabla v| \|_{L^2(\Omega)}^2.
\end{equation*}
$H^1(\Omega)$ is also a Hilbert space when equipped with the (equivalent) $k$-dependent norm \eqref{eq:H1k}, for some $k>0$.
Let $C_0^\infty(\Omega)$ denote the set of $C^\infty$ functions compactly supported in $\Omega$ and let
 $H_0^1(\Omega)$ $:=$ $\text{clos}_{H^1(\Omega)}(C^{\infty}_0(\Omega))$, the closure of $C_0^\infty(\Omega)$ in $H^1(\Omega)$, and let $H^1_0(\Delta, \Omega):=\{u\in H^1_0(\Omega): \Delta u \in L^2(\Omega)\}$. We will also use Sobolev spaces $H^s(\Gamma)$, for $|s|\leq 1$, in the case that $\Gamma$ is the bounded boundary of some Lipschitz domain, defined as in \cite{mclean2000strongly} or \cite{chandler2012numerical}. 
 
 We will use  certain spaces of compactly supported and locally integrable functions. Let $\Lcomp(\Omega)$ denote the space of functions in $L^2(\Omega)$ with bounded support.
 Let $\Lloc(\Omega)$ denote the set of functions on $\Omega$ that satisfy $u|_V\in L^2(V)$ for every bounded, measurable $V\subset \Omega$,
 and let
\begin{eqnarray*}
 H^{1,\mathrm{loc}}(\Omega)&:=& \{u\in \Lloc(\Omega):\nabla u\in (\Lloc(\Omega))^n\}\\ &=&\{u\in \Lloc(\Omega): \chi|_\Omega \,u\in H^1(\Omega), \forall \chi\in C^\infty_0(\R^n)\},\\
H_0^{1,\mathrm{loc}}(\Omega)&:=&\{u\in H^{1,\mathrm{loc}}(\Omega): \chi|_\Omega\, u\in H_0^1(\Omega), \forall \chi\in C^\infty_0(\mathbb{R}^n)\},
\end{eqnarray*}
and
$$
\Hol_0(\Delta, \Omega):=\{u\in \Hol_0(\Omega): \chi|_{\Omega}u\in H^1_0(\Delta, \Omega), \forall \chi\in C_0^\infty(\mathbb{R}^n)\}.
$$

\subsection{A weak formulation of the scattering problem} \label{sec:weak}

Key to our arguments will be a standard reformulation (e.g., \cite[Theorem 2.6.6]{Ned}) of our scattering problem as a variational problem in $\Omega_R=\Omega \cap B_R$, for some $R>R_\Gamma= \max_{x\in \Gamma}|x|$. This variational formulation is set in the closed subspace $V_R\subset H^1(\Omega_R)$ defined by
\begin{equation}
V_R:=\{v|_{\Omega_R}: v\in H_0^1(\Omega) \},
\end{equation}
which we equip with the norm $\|\cdot\|_{H^1_k(\Omega_R)}$.
Note that, where $u\in \Hol_0(\Delta,\Omega)$ is the solution to the scattering problem, $u_R:= u|_{\Omega_R}\in V_R$, for $R>R_\Gamma$. 

For $R>R_\Gamma$ and $k>0$ let $\Gamma_R:= \partial B_R=\{x\in \R^n:|x|=R\}$ and let $\gamma_R:V_R\to H^{1/2}(\Gamma_R)$  denote the standard trace operator. In the appendix we recall that there is a well-defined Dirichlet to Neumann map $\DTN:H^{1/2}(\Gamma_R)\to H^{-1/2}(\Gamma_R)$ for the domain $G_R:= \{x\in \R^n:|x|>R\}$ exterior to $B_R$. This maps $\psi\in H^{1/2}(\Gamma_R)$ to the normal derivative $\partial_r w\in H^{-1/2}(\Gamma_R)$ of the unique solution $w\in H^{1,\mathrm{loc}}(G_R)$ of the homogeneous Helmholtz equation in $G_R:= \{x\in \R^n:|x|>R\}$ that satisfies the radiation condition \eqref{eq:src} and the Dirichlet boundary condition $\gamma_R^+w=\psi$, where $\gamma_R^+:H^{1,\mathrm{loc}}(G_R)\to H^{1/2}(\Gamma_R)$ is the exterior trace operator. Note that if $u\in \Hol_0(\Delta,\Omega)$ satisfies the scattering problem, then $u|_{G_R}$ satisfies this Dirichlet problem with $\psi = \gamma^+_R u = \gamma_R u$, so that 
\begin{equation} \label{eq:DtN2}
\left.\frac{\partial u}{\partial r}\right|_{\Gamma_R} = \DTN \gamma_R u.
\end{equation}
For $R>R_\Gamma$ and $k>0$ define the bounded sesquilinear form $a_{k,R}:V_R\times V_R\rightarrow \mathbb{C}$ by
\begin{equation} \label{eq:ses}
a_{k,R}(v,w):=\int_{\Omega_R} (\nabla v\cdot \nabla \overline{w}-k^2 v\overline{w})\,\rd x- \langle \DTN \gamma_R v, \gamma_R w\rangle, \qquad v,w\in V_R,
\end{equation}
where $\langle\cdot,\cdot\rangle$ denotes the duality pairing on $H^{-1/2}(\Gamma_R)\times H^{1/2}(\Gamma_R)$.

To obtain the variational formulation we choose $R>R_\Gamma$ so that $\supp(f)\subset \Omega_R$. 
We multiply \eqref{eq:he}, restricted to $\Omega_R$, by $\overline{v_R}$, the complex conjugate of some $v_R\in V_R$, and apply the divergence theorem. This leads, using \eqref{eq:DtN2}, to the variational formulation that $u_R:= u|_{\Omega_R}\in V_R$ satisfies
\begin{equation} \label{eq:var}
a_{k,R}(u_R,v_R) = G(v_R), \qquad \forall v_R\in V_R,
\end{equation} 
where the continuous anti-linear functional $G:V_R\to \C$ is given by 
$$
G(v_R):= (f,v_R)_{L^2(\Omega_R)}, \qquad v_R\in V_R,
$$ 
and $(\cdot,\cdot)_{L^2(\Omega_R)}$ is the standard inner product on $L^2(\Omega_R)$. We have shown one part of the following standard result (e.g., \cite{Ned}, \cite[Lemma 2.4]{Sp23}). 

\begin{theorem} \label{thm:equiv} If $u\in H^{1,\mathrm{loc}}_0(\Delta,\Omega)$ satisfies the scattering problem then, for all $R>R_\Gamma$ with $\supp(f)\subset \Omega_R$, $u_R:= u|_{\Omega_R}\in V_R$ satisfies the variational problem \eqref{eq:var}. Conversely, if, for some $R>R_\Gamma$ with $\supp(f)\subset \Omega_R$, $u_R\in V_R$ satisfies \eqref{eq:var} and we set $u(x):=u_R(x)$, $x\in \Omega_R$, $u(x):= w(x)$, $x\in G_R$, where $w$ denotes the unique solution of the Dirichlet problem in $G_R$ with data $\gamma^+_Rw=\gamma_Ru_R$, then   $u\in H^{1,\mathrm{loc}}_0(\Delta,\Omega)$ and satisfies the scattering problem.
\end{theorem}

This result, and that the scattering problem is well-posed, imply that the variational problem \eqref{eq:var} has at most one solution for every $G\in V_R^*$, where $V_R^*$ is the dual space of $V_R$. Existence follows by Fredholm theory by showing that $a_{k,R}(\cdot, \cdot)$ is the sum of coercive and compact sesquilinear forms (see, e.g., the discussion in \cite[\S2.2]{CWHeMoBe:21}). (We say that a sesquilinear form $b(\cdot,\cdot)$ on $V_R$ is compact if the associated operator $B:V_R\to V_R^*$, defined by $Bv(w):=b(v,w)$, $v,w\in V_R$, is compact.) We write  $a_{k,R}(\cdot,\cdot)=\widetilde a_{k,R}(\cdot,\cdot) + b_{k,R}(\cdot,\cdot)$, where $b_{k,R}(v,w) := -2k^2(v,w)_{L^2(\Omega_R)}$, $v,w\in V_R$, so that $b_{k,R}(\cdot,\cdot)$ is compact by the compactness of the embedding of $V_R$ in $L^2(\Omega_R)$. Further,
\begin{equation} \label{eq:coer}
\mathrm{Re}\, \widetilde a_{k,R}(v,v) = \int_{\Omega_R} (|\nabla v|^2 + k^2 |v|^2)\,\rd x- \mathrm{Re}\, \langle \DTN \gamma_R v, \gamma_R v\rangle \geq \|v\|_{H^1_k(\Omega_R)}^2,
\end{equation}
 for all $v\in V_R$, by \eqref{eq:21one} in Lemma \ref{lem:2.1}, so that $\widetilde a_{k,R}(\cdot, \cdot)$ is coercive. We note that \eqref{eq:coer} and this decomposition of $a_{k,R}(\cdot,\cdot)$ are standard for $n=2,3$ (e.g., \cite[p.~1441]{CWMonk2008}, \cite[Lemma 2.6]{Sp23}). The inequality  \eqref{eq:coer} may be stated for the first time here for $n>3$, since, to the best of our knowledge, the second inequality in \eqref{eq:21one} is not previously established except for $n=2,3$.

 Let $\beta_{k,R}$ denote the {\em inf-sup constant} of $a_{k,R}(\cdot,\cdot)$, defined as
\begin{equation} \label{eq:infsup}
\beta_{k,R}:=\inf_{0\neq v\in V_R}\sup_{0\neq w\in V_R} \frac{\lvert a_{k,R}(v,w)\rvert}{\lVert v\rVert_{H^1_k(\Omega_R)}\lVert w\rVert_{H^1_k(\Omega_R)}}.
\end{equation}
 The well-posedness of \eqref{eq:var} implies that $\beta_{k,R}>0$, for $k>0$, $R>R_\Gamma$ (see, e.g., \cite[Remark 2.20]{Ih:98}). Further, by a standard property of bounded sesquilinear forms  (e.g., \cite[Theorem 2.1.44]{SaSch}), $\beta_{k,R}^{-1}=\widetilde C_{k,R}$, where $\widetilde C_{k,R}$ is the smallest constant such that
 \begin{equation} \nonumber
 \|u_R\|_{H^1_k(\Omega_R)} \leq \widetilde C_{k,R} \|G\|_{V_R^*}, \qquad G\in V_R^*,
 \end{equation}
 where $u_R\in V_R$ denotes the solution to \eqref{eq:var} for the  functional $G$ (cf.~\cite[Lemma 3.3]{CWMonk2008}).
 Thus, and by \cite[Lemma 3.4]{CWMonk2008}, we have (see \cite[Lemma 5.1]{chandler2020high}) that
\begin{equation} \label{eq:is}
kC_{k,R}'/\min(1,kc_R) \leq \beta_{k,R}^{-1}= \widetilde C_{k,R} \leq  1+2k C_{k,R}', \qquad k>0, \quad R>R_\Gamma,
\end{equation}
where $C_{k,R}'$ is the cut-off resolvent norm \eqref{eq:Cdashdef} and\footnote{We remark that $c_R=\lambda_Z^{-1/2}$, where $\lambda_Z$ is the first eigenvalue for the Laplacian with Dirichlet boundary conditions on $\Gamma$, Neumann on $\Gamma_R$ (e.g., \cite[Equation (3.1.21)]{levitin2023topics}), and that $c_R\geq \sup_{0\neq v\in H_0^1(B_{R'})}\|v\|_{L^2(B_{R'})}/\| |\nabla v| \|_{L^2(B_{R'})} = R'/j_{n/2-1,1}$, where $R'$ denotes the radius of the largest ball contained in $\Omega_R$, and $j_{\nu,1}$ denotes the first positive zero of the Bessel function $J_\nu$ (see, e.g., \cite[Exercise 1.2.21, Theorem 3.1.9]{levitin2023topics} for the last equality).}
$$
c_R := \sup_{0\neq v\in V_R}\frac{\|v\|_{L^2(\Omega_R)}}{\| |\nabla v| \|_{L^2(\Omega_R)}}.
$$ (The proof of \cite[Lemma 3.4]{CWMonk2008} is for dimensions $n=2,3$, but carries over to all $n\geq 2$ now that \eqref{eq:21one} and \eqref{eq:coer} are available for $n>3$.)

Further, if $u$ is the solution of the scattering problem and $u_R:= u|_{\Omega_R}$, it follows from Theorem \ref{thm:equiv} and \eqref{eq:coer} that
$$
\mathrm{Re}\, (f,u_R)_{L^2(\Omega_R)} = \mathrm{Re}\,  a_{k,R}(u_R,u_R) \geq \|u_R\|_{H^1_k(\Omega_R)}^2 - 2k^2 \|u_R\|_{L^2(\Omega_R)}^2, \;\; k>0, \; R>R_\Gamma.
$$
From this inequality we obtain, by Cauchy-Schwarz, that \eqref{eq:equiv} holds. This, together with \eqref{eq:is}, implies in turn that
\begin{equation}\label{infsup}
\beta_{k,R}^{-1}\leq 1+2k\sqrt{2k^2(C_{k,R})^2 +C_{k,R}}, \qquad k>0, \quad R>R_\Gamma.
\end{equation}

Note that bounds on any one of the inf-sup constant $\beta_{k,R}$ and the resolvent norms $C_{k,R}$ and $C_{k,R}'$ can be translated into bounds on the other two quantities, via \eqref{eq:equiv} and \eqref{eq:is}. In particular, we will use \eqref{infsup}, derived from these inequalities, in our proof of Theorem \ref{thm1} below. 

\subsection{The multiplicative trace inequality in a cube} \label{sec:mult}
Our arguments, specifically our proof of Theorem \ref{thm1}, need the following special case, with explicit constants, of the standard multiplicative trace inequality (e.g., \cite[Theorem 1.5.1.10, last formula on p. 41]{Grisvard}). In this lemma and its proof $\gamma:H^1(\mO)\to L^2(\partial \mO)$ is the standard trace operator.
\begin{lemma} \label{lem:mult}
Let $a>0$, $\mathcal{O}:=(0,a)^n\subset \mathbb{R}^n$ and $v\in H^1(\mathcal{O})$ with $\mathrm{supp}(\gamma\, v)\subset [0,a]^{n-1}\times \{0\}$. Then
\begin{equation}\label{multtrace2}
\lVert \gamma \,v\rVert_{L^2(\partial \mathcal{O})}^2\leq  2\lVert v\rVert_{L^2(\mathcal{O})}\lVert \nabla v\rVert_{L^2(\mathcal{O})} \leq k^{-1}\|v\|^2_{H^1_k(\mathcal{O})}, \quad k>0.
\end{equation}
\end{lemma}
\begin{proof}
The proof here adapts the argument in \cite{BrSc:08}, where an analogous  bound is shown  for the unit disk in $\mathbb{R}^2$ (and cf.~\cite{Grisvard}). 
Suppose that $v\in C^1(\overline{\mathcal{O}})$. Using coordinates $x=(y,t)$, where $y\in [0,a]^{n-1}\subset \mathbb{R}^{n-1}$ and $t\in \mathbb{R}$,
\begin{equation*}
\lvert v(y,0)\rvert^2=-\int_0^a \frac{\partial}{\partial t}\lvert v(y,t)\rvert^2\,\rd t=-2\mathrm{Re}\,\bigg(\int_0^a \overline{v(y,t)}\frac{\partial v}{\partial t}(y,t)\,\rd t\bigg).
\end{equation*}
Since $\mathrm{supp}(\gamma\, v)\subset [0,a]^{n-1}\times \{0\}$, integrating the above identity over $[0,a]^{n-1}$ gives that
\begin{equation*}
\begin{split}
\lVert \gamma\,v\rVert_{L^2(\partial \mathcal{O})}^2&= -2\mathrm{Re}\,\bigg(\int_{[0,a]^{n-1}}\int_0^a \overline{v(y,t)} \frac{\partial v}{\partial t}(y,t)\,\rd t\rd y\bigg)\\
&\leq 2\int_{\mathcal{O}} \lvert v(x)\rvert \lvert \nabla v(x)\rvert\,\rd x \leq 2 \lVert v\rVert_{L^2(\mathcal{O})} \lVert \nabla v\rVert_{L^2(\mathcal{O})}.
\end{split}
\end{equation*}
 The first inequality in \eqref{multtrace2} follows from the above inequality by the density of $C^1(\overline{\mathcal{O}})$ in  $H^1(\mathcal{O})$ and the continuity of $\gamma$. The second inequality is elementary.
\end{proof}
\section{Proofs of the main results} \label{sec:main2}
\subsection{Proof of Theorem \ref{thm1}} \label{sec:thm1}
We turn now to the proof of Theorem \ref{thm1}, one of our main results. In the proof we obtain, for each $j\in \N$ and $R>R_\Gamma$, the desired lower bound on $C_{k_j,R}$, the norm of the cut-off resolvent, by constructing a quasimode $u_j\in V_R$ for the weak formulation \eqref{eq:var}, by which we mean a $u_j\in V_R$ for which 
\begin{equation} \label{eq:quant}
\sup_{0\neq v\in V_R}\frac{|a_{k_j,R}(u_j,v)|}{\|u_j\|_{H^1_{k_j}(\Omega_R)}\|v\|_{H^1_{k_j}(\Omega_R)}}
\end{equation}
is sufficiently small.  The smallness of this quantity is guaranteed by our choice of $u_j$ and our choice \eqref{gaplength} of $\varepsilon_j$, and established with the help of the mutiplicative trace inequality, Lemma  \ref{lem:mult}. The upper bound we obtain for \eqref{eq:quant} implies an upper bound on the inf-sup constant $\beta_{k_j,R}$, which leads, via \eqref{infsup}, to a lower bound on $C_{k_j,R}$. Our construction of the quasimode $u_j$ is simply to extend the first eigenfunction of the Dirichlet Laplacian in the cube $O_j$ by zero to $\Omega_R$.

As we will see below the above method of argument leads to a short proof for Theorem \ref{thm1} which provides the explicit expression \eqref{gaplength} for the $\varepsilon_j$. An anonymous reviewer has encouraged us to point out a variation on our argument and construction that leads to a modified version of Theorem \ref{thm1} with a different expression for $\varepsilon_j$. This alternative argument starts by smoothing $u_j$ to obtain a $v_j\in \Hol_0(\Delta, \Omega)$, supported in $\overline{O_j}$, that is a solution of the scattering problem with a source term $f= f_j$, also supported in $\overline{O_j}$, so that $C_{k_j,R}\geq \|v_j\|_{L^2(\Omega_R)}/\|f_j\|_{L^2(\Omega_R)}$, by \eqref{eq:resest}. Given any increasing sequence $(a_j)$ it is possible to guarantee $\|v_j\|_{L^2(\Omega_R)}/\|f_j\|_{L^2(\Omega_R)}\geq a_j$, so that $C_{k_j,R}\geq a_j$,  by taking $\varepsilon_j$ sufficiently small, so that $(v_j)$ is a sequence of quasimodes for the original scattering problem, in the sense, e.g., of  \cite[\S7.3]{DyZw}, just as $(u_j)$ is a sequence of quasimodes for its weak formulation. 

In more detail, choose $\chi \in C^2(\R^n)$ such that $\chi(x)= 0$ for $|x|\leq 2\sqrt{n}$, $\chi(x)=1$ for $|x|\geq 3\sqrt{n}$. Assuming, without loss of generality, that $t_j=0$, define the $j$th quasimode $v_j$ by $v_j:= \chi_ju_j$, where $\chi_j(x) := \chi(x/(\varepsilon_j\ell_j))$, $x\in \R^n$, and $u_j$ is defined by \eqref{eq:qm} below. Then $v_j$ vanishes in a neighbourhood of $\overline{\partial \mO_j\setminus \Gamma_j}$ so that it is clear that $v_j\in \Hol_0(\Delta, \Omega)$, and $f_j := -(\Delta +k_j^2)v_j = -u_j\Delta \chi_j - 2 \nabla u_j \cdot \nabla \chi_j$. Further, as $\varepsilon_j\to 0$ it holds uniformly for $x\in B_{3\sqrt{n} \varepsilon_j\ell_j}\supset \supp(f_j)$ that $\nabla \chi_j(x)=O(\varepsilon_j^{-1})$, $\Delta \chi_j(x) = O(\varepsilon_j^{-2})$, and, noting the explicit form \eqref{eq:qm} for $u_j$, that $u_j(x)=O(\varepsilon^n_j)$, and $\nabla u_j(x) = O(\varepsilon_j^{n-1})$, so that $f_j(x)=O(\varepsilon_j^{n-2})$. Thus, and since $|\supp(f_j)|_n = O(\varepsilon_j^n)$, $\|f_j\|_{L^2(\Omega_R)}=O(\varepsilon_j^{(3n-4)/2})$ as $\varepsilon_j\to 0$. Since also $\|v_j\|_{L^2(\Omega_R)}\to \|u_j\|_{L^2(\Omega_R)}$, which is independent of $\varepsilon_j$, as $\varepsilon_j\to 0$, it follows that there is a constant $C_j>0$ such that
$$
C_{k_j,R}\geq \|v_j\|_{L^2(\Omega_R)}/\|f_j\|_{L^2(\Omega_R)}\geq C_j \varepsilon_j^{-(3n-4)/2},
$$
for all sufficiently small $\varepsilon_j$ (cf.~\eqref{eq:lbnew}). 

With a little more care over these calculations one sees that the above inequalities hold if $\varepsilon_j\leq 1/(6\sqrt{n})$ and  with $C_j=k_j^{-2}C$, for some $C>0$ which depends only on $n$ and the choice of $\chi$. Thus $C_{k_j,R} \geq a_j$ if 
$$
\varepsilon_j\leq \min\left(\frac{1}{6\sqrt{n}},\left(\frac{k_j^2a_j}{C}\right)^{-2/(3n-4)}\right).
$$
Further calculation should yield an expression  for $C$ as a function of $n$ for some specific choice of $\chi$. The above formula can be compared to \eqref{gaplength}, which requires that $\varepsilon_j= O\left((k_j^2a_j)^{-2/(3n-3)}\right)$ as $j\to\infty$, a slightly slower rate of decay\footnote{It should  also be possible to obtain a lower estimate for $C_{k_j,R}$ in the limit as $\varepsilon_j\to 0$ by using rigorous Helmholtz resonator results on the asymptotics of solutions to the scattering problem, of resonances, and of associated generalised eigenfunctions, though this estimate may not have completely explicit constants.  See, in particular, \cite[\S1.4-1.5,\S2.5]{Gady}, \cite{Gady2}.}.

\begin{proof}[Proof of Theorem \ref{thm1}]
Let $\Gamma$ be as in Theorem \ref{thm1}, $R>R_\Gamma$, and $j\in \N$. Recall that the sidelength of the $j$th cube $\mO_j$ is $\ell_j=\pi \sqrt{n}k_j^{-1}$. Without loss of generality, suppose that $t_j=0$ so that $\mO_j = \ell_j\mO$,
and define the test function $u_j\in V_R$ by
\begin{equation} \label{eq:qm}
u_j(x)=\begin{cases}
\prod_{i=1}^n \sin\left(\frac{k_j}{\sqrt{n}} x_i\right), \hspace{0.2cm} x=(x_1,...,x_n)\in \mathcal{O}_j,\\
0, \hspace{0.2cm} x\in \Omega_R\setminus \mathcal{O}_j.
\end{cases}
\end{equation}
Then the sesquilinear form \eqref{eq:ses}, with $u_j$ as the first argument and $v\in V_R$ as the second, reduces to
\begin{equation}
a_{k,R}(u_j,v)=\int_{\mathcal{O}_j} (\nabla u_j\cdot \nabla \overline{v}-k^2 u_j\overline{v})\,\rd x.
\end{equation}
Furthermore, by Green's theorem and the fact that $(\Delta+k_j^2)u_j=0$ in $\mathcal{O}_j$ and $\gamma\,v=0$ on $\Gamma_j$, where $\gamma:H^1(\mO_j)\to L^2(\partial \mO_j)$ is the trace operator, it follows that
\begin{equation}\label{harmony}
|a_{k_j,R}(u_j,v)|=\left|\int_{\partial\mathcal{O}_j\setminus \Gamma_j} \partial_\nu u_j \overline{\gamma\, v}\,\rd s\right| \leq \lVert \partial_\nu u_j\rVert_{L^2(\partial\mathcal{O}_j\setminus \Gamma_j)}\lVert \gamma\, v\rVert_{L^2(\partial \mathcal{O}_j)}, 
\end{equation}
by Cauchy-Schwarz.
Now
\begin{equation}
\partial_\nu u_j(x)=-\frac{k_j}{\sqrt{n}}\prod_{i=1}^{n-1}\sin\left(\frac{k_j}{\sqrt{n}}x_i\right), \quad x\in \partial\mathcal{O}_j \setminus \Gamma_j.
\end{equation}
Explicitly computing the $L^2$ norm of $\partial_\nu u_j$,
\begin{eqnarray*}
\lVert \partial_\nu u_j\rVert^2_{L^2(\partial\mathcal{O}_j\setminus \Gamma_j)} &=&\frac{1}{4^2}\bigg(\frac{\sqrt{n}}{4k_j}\bigg)^{n-3}\bigg(\frac{2k_j\ell_j\varepsilon_j}{\sqrt{n}}-\sin\left(\frac{2k_j\ell_j\varepsilon_j}{\sqrt{n}}\right)\bigg)^{n-1}\\
&\leq &\frac{1}{4^2}\bigg(\frac{\sqrt{n}}{4k_j}\bigg)^{n-3} \bigg(\frac{1}{3!}\bigg(\frac{2k_j \ell_j\varepsilon_j}{\sqrt{n}}\bigg)^3\bigg)^{n-1} = \frac{(k_j\ell_j\varepsilon_j)^{3(n-1)}}{3^{n-1}k_j^{n-3}n^n},
\end{eqnarray*}
where we have used the basic estimate $t-\sin(t)\leq \frac{t^3}{3!}$ for $t\geq 0$, which follows from Taylor's theorem with integral remainder. Another straightforward computation yields
\begin{equation*}
\lVert u_j\rVert_{H_{k_j}^1(\Omega_R)}^2=\frac{k_j^2 \ell_j^n}{2^{n-1}}.
\end{equation*}
We combine these computations with \eqref{harmony} to get a bound for the inf-sup constant \eqref{eq:infsup}, that
\begin{eqnarray*}
\beta_{k_j,R}&\leq& \sup_{0\neq v\in V_R}\frac{|a_{k_j,R}(u_j,v)|}{\|u_j\|_{H^1_{k_j}(\mathcal{O}_j)}\|v\|_{H^1_{k_j}(\mathcal{O}_j)}}\\
&\leq &\frac{2^{(n-1)/2}k_j^{1/2}(k_j\ell_j\varepsilon_j)^{3(n-1)/2}}{3^{(n-1)/2}(nk_j\ell_j)^{n/2}}\sup_{0\neq v\in V_R}\frac{\| \gamma\, v\|_{L^2(\partial \mathcal{O}_j)}}{\lVert v\rVert_{H^1_{k_j}(\mathcal{O}_j)}}.
\end{eqnarray*}
 Using Lemma \ref{lem:mult} and that $\ell_j = \pi\sqrt{n} k_j^{-1}$, we obtain that
\begin{equation} \label{eq:lbnew}
\beta_{k_j,R}\leq c_n \varepsilon_j^{3(n-1)/2},  \quad \mbox{where} \quad c_n := \frac{2^{(n-1)/2}\pi^{n-3/2}}{3^{(n-1)/2}n^{3/4}}.
\end{equation}
Next we see that the above inequality, coupled with \eqref{infsup} which relates the inf-sup constant and the cut-off resolvent norm $C_{k,R}=\lVert \chi_R R(k)\chi_R\rVert_{L^2(\Omega)\to L^2(\Omega)}$, yields
\begin{equation}\label{explicitineq}
c_n^{-1}\varepsilon_j^{-3(n-1)/2}\leq \beta_{k_j,R}^{-1} \leq 1+2k_j\sqrt{2k_j^2 (C_{k_j,R})^2+C_{k_j,R}}.
\end{equation}
Now \eqref{gaplength} implies that
\begin{equation}
c_n^{-1}\varepsilon_j^{-3(n-1)/2}=n^{3/4}\pi^{1/2}\bigg(1+2k_j\sqrt{2k_j^2 a_j^2 +a_j}\bigg).
\end{equation}
Inserting this into \eqref{explicitineq} and simplifying a little, we see that
\begin{equation}
2k_j^2a_j^2+a_j< 2k_j^2(C_{k_j,R})^2+ C_{k_j,R}.
\end{equation}
It follows that $\lVert \chi_R R(k_j)\chi_R \rVert_{L^2(\Omega)\rightarrow L^2(\Omega)}= C_{k_j,R} > a_j$.
\end{proof}
\subsection{Cube packing and the proofs of Theorems \ref{thm2} and \ref{thm2v}} \label{sec:p19}
We turn now to the proofs of Theorems \ref{thm2} and Theorem \ref{thm2v}, i.e.~to showing that conditions (i)-(iii) in Theorem \ref{thm1} can be realised by a certain configuration of the sequence of cubes, $(\mO_j)_{j\in \N}$, under the mild assumption \eqref{eq:ellb} on the rate of decrease of the sidelengths of these cubes in the case of Theorem \ref{thm2}, under the (marginally) stronger assumption \eqref{eq:ellbv} in the case of Theorem \ref{thm2v}. The proof of each theorem is constructive, providing a specification of the sequence of translations $(t_j)_{j\in \N}$ whose existence the theorem asserts. The configurations of the cubes differ significantly between the two proofs. In particular the configuration in the proof of Theorem \ref{thm2v} is such that also \eqref{eq:onep} holds (see \eqref{eq:onep2} below), relevant to our discussion of Question 2 in \S\ref{sec:fq}. 

Our first result is a version of \cite[Theorem 1]{MM68} which is a result on the packing of finite and infinite sequences of cubes inside a larger cube (or more generally, within some sufficiently large rectangular parallelepiped). We include our own version of a proof of this (beautiful) result with the aim of making explicit within the proof the choice for the translations $(t_j)_{j \in \mI}$. While most of the arguments of our proof come from \cite{MM68} it may be helpful, for comprehension, to read our proof alongside that of \cite{MM68}. In particular, we make the induction step of the proof in \cite{MM68} explicit, make clear that the base case of the induction is $n=1$, and deal more explicitly with the case when the index set $\mI$ is finite (which is needed in the induction hypothesis even if one's interest is only in infinite sequences of cubes).

\begin{theorem}[{\cite{MM68}}] \label{thm:cube}
Suppose $n\in \N$, that $\mI\subset \N$ is an index set, either $\mI=\{1,\dots,N\}$, for some $N\in \N$, or $\mI=\N$, that the sequence $(\ell_j)_{j\in \mI}\subset (0,\infty)$ is non-increasing, with $V := \sum_{j\in \mI}\ell_j^n < \infty$, and that $a> \ell_1 + (V-\ell_1^n)^{1/n}$. 
Then there exists a choice of the translations $(t_j)_{j\in \mI}\subset \R^n$ such that the cubes  $(\mO_j)_{j\in \mI}$, defined by
$\mO_j := \ell_j \mO+t_j$, for $ j\in \mI$,
have pairwise disjoint closures, i.e., $\overline{\mO_i}\cap \overline{\mO_j}=\emptyset$, for $i,j\in \mI$, $i\neq j$, and are such that $\mO_\infty := \bigcup_{j\in \mI} \mO_j\subset a\mO$. 
\end{theorem}
\begin{proof} 
The proof is by induction on $n$. Clearly the claim is true for $n=1$. Suppose that $n\in \N$ with $n\geq 2$, and, as the induction hypothesis, that the claim of the theorem is true with $n$ replaced by $n-1$. Suppose further that $\mI$,  $(\ell_j)_{j\in \mI}$, $a$, $V$,  $(\mO_j)_{j\in \mI}$ are as in the statement of the theorem. To complete the inductive step we have to show that we can then choose the translations $(t_j)_{j\in \mI}\subset \R^n$ to satisfy the required constraints on $(\mO_j)_{j\in \mI}$ of the theorem. 

For $j\in \mI$ we write $t_j = (\tilde t_j,t_{j,n})$ where $\tilde t_j\in \R^{n-1}$ and $t_{j,n}\in \R$, and write $\mO_j = \widetilde \mO_j \times [t_{j,n},t_{j,n}+\ell_j]$, where $\widetilde \mO_j :=\ell_j\widetilde \mO + \tilde t_j$ and $\widetilde \mO := \{(x_1,\ldots,x_{n-1})\in \R^{n-1}:0<x_i<1 \mbox{ for } i=1,\ldots,n-1\}$, so that $\widetilde \mO$ and $\widetilde \mO_j$, $j\in \mI$, are $(n-1)$-dimensional cubes. We choose a sequence $(d_i)_{i\in \N}\subset (0,\infty)$ such that
$$
\epsilon:= \sum_{i=1}^\infty d_i< a- \ell_1 - (V-\ell_1^n)^{1/n}.
$$
and note that
\begin{equation} \label{eq:equi}
a> \epsilon + \ell_1 + (V-\ell_1^n)^{1/n} \qquad \Leftrightarrow \qquad a>\epsilon+\ell_1 \quad \mbox{and} \quad V < \ell_1^n+(a-\epsilon-\ell_1)^n.
\end{equation}

We make the following recursive definition of a sequence $(B_i)_{i\in \mJ}\subset \mI$, where $\mJ=\{1,\ldots,M\}$, for some $M\in \N$, or $\mJ=\N$ (in which case we put $M=+\infty$). We set $B_1=1$. If $i\in \mJ$ then $i+1\in \mJ$ as long as 
\begin{equation} \label{eq:fin}
\sum_{\overset{\scriptstyle j\in \mI}{\scriptstyle j\geq B_i}} \ell_j^{n-1} \geq \ell_{B_i}^{n-1} + (a-\epsilon-\ell_{B_i})^{n-1},
\end{equation}
in which case we define $B_{i+1}>B_i$ as the unique positive integer satisfying
\begin{equation} \label{eq:Bdef}
\sum_{j=B_i}^{B_{i+1}-1}\ell_j^{n-1} < \ell_{B_i}^{n-1} + (a-\epsilon-\ell_{B_i})^{n-1} \leq \sum_{j=B_i}^{B_{i+1}}\ell_j^{n-1}.
\end{equation}
If \eqref{eq:fin} does not hold for some $i\in \mJ$, then $\mJ$ is finite, $M=i$, and 
\begin{equation} \label{eq:Bdef2}
\sum_{j\in I_M}\ell_j^{n-1} < \ell_{B_M}^{n-1} + (a-\epsilon-\ell_{B_M})^{n-1},
\end{equation}
where $I_M :=\{j\in \mI: j\geq B_M\}$. For $i\in \mJ$ with $i<M$ we put $I_i := \{B_i,\ldots,B_{i+1}-1\}$. 

We
place the cubes $\mO_j$ in layers parallel to the hyperplane $x_n=0$, indexed by the set $\mJ$, placing the cubes $\{\mO_j:j\in I_i\}$ in layer $i$, for $i\in \mJ$, setting
\begin{equation} \label{eq:tjn}
t_{j,n} = H_i, \qquad j\in I_i, \quad i\in \mJ, 
\end{equation} 
where 
$$
H_1 := 0, \quad H_{i} := H_{i-1} + \ell_{B_{i-1}} + d_{i-1}, \quad \mbox{for } i\in \mJ, \;\; i\neq 1. 
$$
Note that \eqref{eq:tjn}, and the definition of $H_i$ for $i\in \mJ$, imply that $\overline{\mO_j}\cap\overline{\mO_{j'}}=\emptyset$ if $j\in I_i$, $j'\in I_{i'}$, and $i,i'\in \mJ$ with $i< i'$. For if $x=(x_1,\ldots,x_n)\in \overline{\mO_j}$ and $x'=(x_1',\ldots,x_n')\in \overline{\mO_{j'}}$, then $x_n'\geq t_{j',n}=H_{i'}\geq H_{i+1}=H_i+\ell_{B_{i}}+d_i\geq t_{j,n}+\ell_j+d_i>x_{n}$.

Further, by the inductive hypothesis, \eqref{eq:equi}, and the left-hand inequality in \eqref{eq:Bdef} (and \eqref{eq:Bdef2} in the case that $\mJ$ is finite),  for each $i\in \mJ$ there exists a choice of the translations $(\tilde t_j)_{j\in I_i}\subset \R^{n-1}$ such that the cubes  $(\widetilde \mO_j)_{j\in I_i}$ have pairwise disjoint closures and $\bigcup_{j\in I_i} \widetilde \mO_j\subset (a-\epsilon)\widetilde \mO \subset a\widetilde \mO$. Making, for each $i\in \mJ$, this choice for $(\tilde t_j)_{j\in I_i}$, it follows that $\overline{\mO_j}\cap\overline{\mO_{j'}}=\emptyset$ if $j,j'\in I_i$ with $j\neq j'$. Thus the cubes $(\mO_j)_{j\in \mI}$ have pairwise disjoint closures. Further, our choice for $(\tilde t_j)_{j\in \mI}$ ensures that $\widetilde \mO_j \subset a\widetilde\mO$   for $j\in \mI$, so that $\mO_\infty \subset a\mO$ as long as $0\leq t_{j,n} \leq a-\ell_j$, for each $j\in \mI$, which holds if and only if $H_i+\ell_{B_i} \leq a$, for all $i\in \mJ$, which holds if $\sum_{i\in \mJ} (\ell_{B_i}+d_i) \leq a$. 

To see that this last bound holds, let $V_i$ be the volume of the cubes in layer $i$. Then, for $i\in \mJ$ with $i<M$ (so for all $i\in\mJ$ in the case that $\mJ=\N$ and $M=+\infty$),
$$
V_i = \sum_{j=B_i}^{B_{i+1}-1} \ell_j^n \geq \ell_{B_i}^n + \ell_{B_{i+1}}\sum_{j=B_i+1}^{B_{i+1}-1} \ell_j^{n-1} \geq \ell_{B_i}^n + \ell_{B_{i+1}}(a-\ell_{B_i}-\epsilon)^{n-1} - \ell_{B_{i+1}}^n,
$$
by the right-hand inequality of \eqref{eq:Bdef}. Noting \eqref{eq:equi} and that $\ell_{B_i}\leq \ell_{B_1}=\ell_1$, for $i\in \mJ$, and, in the case that $\mJ$ is finite, that $V_M\geq \ell_{B_M}^n$, it follows that 
$$
\ell^n_1 + (a-\ell_1-\epsilon)^{n}> V = \sum_{i\in \mJ} V_i \geq \ell_1^n + (a-\epsilon-\ell_1)^{n-1}\sum_{\overset{\scriptstyle i\in \mJ}{\scriptstyle i<M}} \ell_{B_{i+1}}.
$$
Thus
$$
\sum_{\overset{\scriptstyle i\in \mJ}{\scriptstyle i<M}} \ell_{B_{i+1}} \leq a-\ell_1 -\epsilon,
$$
so that $\sum_{i\in \mJ} (\ell_{B_{i}}+d_i) \leq a$.
\end{proof}

\begin{algorithm}
\caption{{\sc ComputeTranslations}}
\label{alg:trans}
\begin{algorithmic}[1]
\STATE \textbf{Inputs:} {\em The space dimension $n\in \N$; a sequence of cube side-lengths $(\ell_j)_{j\in \mI}\subset (0,\infty)$ (where $\mI:=\{N_1,N_1+1,\ldots\}$ for some $N_1\in \N$, or $\mI:=\{N_1,\ldots,N_2\}$, for some $N_1,N_2\in \N$ with $N_2\geq N_1$) which satisfy $V:=\sum_{j\in \mI}\ell_j^n<\infty$ and $\ell_{j'}\leq \ell_j$, if $j,j'\in \mI$ with $j'\geq j$; a larger cube side-length $a > \ell_{N_1} + (V-\ell_{N_1}^n)^{1/n}$.}
\STATE \textbf{Output:} {\em A sequence of translation vectors $(t_j)_{j\in \mI}\subset \R^n$ such that $\overline{\mO_j}\cap \overline{\mO_{j'}}=\emptyset$, if $j,j'\in \mI$ with $j\neq j'$, and $\mO_\infty\subset a\mO$, where $(\mO_j)_{j\in \mI}$ is the sequence of cubes defined by $\mO_j := \ell_j \mO + t_j$, $j\in \mI$, and  $\mO_\infty := \bigcup_{j\in \mI} \mO_j$, which has total volume $V$.}
\STATE{Choose $(d_i)_{i\in\N}\subset(0,\infty)$ such that
$\epsilon:=\sum_{i=1}^\infty d_i<a-\ell_{N_1}
-\bigl(V-\ell_{N_1}^n\bigr)^{1/n}$.}
\IF{$n=1$}
\STATE $j:= N_1$, $t_j := 0$
\WHILE{$j+1\in \mI$}
\STATE{$j:= j+1$}
\STATE{$t_j:= t_{j-1}+\ell_{j-1} + d_j$}
\ENDWHILE
\ELSE
\STATE{$i:= 1$, $B_i := N_1$, $\mJ := \{1\}$}
\WHILE{equation \eqref{eq:fin} holds}
\STATE{$\mJ := \mJ\cup\{i+1\}$}
\STATE{Define $B_{i+1}>B_i$ to be the unique integer satisfying \eqref{eq:Bdef}}
\STATE{$i:= i+1$}
\ENDWHILE
\STATE{Define $(I_i)_{i\in \mJ}\subset \mI$ and $(H_i)_{i\in \mJ}\subset [0,a)$ as in the proof of Theorem \ref{thm:cube}}
\FOR{$i\in \mJ$}
\STATE{$t_{j,n}:= H_i$, for $j\in I_i$}
\STATE{Compute $(\tilde t_j)_{j\in I_i}\subset \R^{n-1}$ by calling {\sc ComputeTranslations} with inputs $n-1$, $(\ell_j)_{j\in I_i}$, and $a$, noting that $a$ satisfies the requirement of the algorithm by the equivalence \eqref{eq:equi} and \eqref{eq:Bdef} (which implies \eqref{eq:Bdef2} if $\mJ$ is finite)}
\ENDFOR
\FOR{$j\in \mI$}
\STATE{$t_j := (\tilde t_j, t_{j,n})$}
\ENDFOR
\ENDIF
\end{algorithmic}
\end{algorithm}

\begin{remark}[Algorithm \ref{alg:trans} and the translations $(t_j)_{j\in \mI}$ in Theorem \ref{thm:cube}] \label{rem:alg}
Using the notations of the proof of Theorem \ref{thm:cube} it is clear that the proof of that theorem provides in \eqref{eq:tjn} a specification for the last component $t_{j,n}\in \R$ of $t_j$, for each $j\in \mI$.  That appropriate choices exist for the remaining components $\tilde t_j\in \R^{n-1}$ follows from the inductive hypothesis applied to each of the sequences $(\ell_j)_{j\in I_i}$, for $i\in \mJ$. This recursive construction of the translations $(t_j)_{j\in \mI}$ is captured in Algorithm \ref{alg:trans}. This algorithm specifies $(t_j)_{j\in \mI}$ for every index set $\mI$, finite or infinite, and  can be converted into code that runs in finite time in the case that $\mI$ is finite (Figure \ref{fig:Gamma} was produced with such a code). We leave to the reader a proof by induction of the correctness of this algorithm, which proof mirrors, step by step, that of Theorem \ref{thm:cube}.
\end{remark}

\begin{remark}[The number of layers produced by Algorithm \ref{alg:trans}] \label{rem:layers}
The construction in the proof of Theorem \ref{thm:cube}, captured in Algorithm \ref{alg:trans}, arranges the cubes $(\mO_j)_{j\in\mI}$ in layers, indexed by the set $\mJ$. It is easy to see from the construction of $\mJ$ (see the discussion around equations  \eqref{eq:fin}-\eqref{eq:Bdef2}), that $\mJ$ is finite if and only if $V_{n-1} :=\sum_{j\in \mI} \ell_j^{n-1} < \infty$, in which case the number of layers is $M$, the cardinality of $\mJ$, with infinitely many cubes in the $M$th layer if $\mI$ is not finite. In particular, if the inputs to Algorithm \ref{alg:trans} are such that $V_{n-1}<\infty$ and $a>\ell_{N_1} + (V_{n-1}-\ell^{n-1}_{N_1})^{1/(n-1)}$, then $M=1$, so that all the cubes are placed in a single layer. As an example of this last possibility, if $n=2$ and the sequence $(k_j)_{j\in \N}$ is as in Figure \ref{fig:GammaV}, in which case, by the formula in Theorem \ref{thm1}, $\ell_j= \pi\sqrt{2}/k_j=(\pi\sqrt{2}/4)(j+4)^{-6/5}$, $j\in \N$, then Algorithm \ref{alg:trans} arranges the cubes in the single layer shown in Figure \ref{fig:GammaV} if the input $a$ is sufficiently large and $(d_i)_{i\in \N}$, in line 3 of the algorithm,  is chosen as in the caption of Figure \ref{fig:GammaV}.
\end{remark}

\begin{proof}[Proof of Theorem \ref{thm2}]
Let $(\varepsilon_j)_{j\in \N}$, $(\ell_j)_{j\in \N}$, and $a$ be chosen as in the statement of the theorem.  With $(\mO_j)_{j\in \N}$ and $(\Gamma_j)_{j\in \N}$ as defined in the theorem we need to choose a sequence of translations  $(t_j)_{j\in \N}\subset \R^n$, i.e., an arrangement of the cubes $(\mO_j)_{j\in \N}$, so that (i)-(iii) in Theorem \ref{thm1} are satisfied. Let $(t_j)_{j\in \N}$ be the sequence whose existence is guaranteed by the construction in Theorem \ref{thm:cube}, i.e., the sequence defined (see Remark \ref{rem:alg})  by a call to the algorithm {\sc ComputeTranslations} (Algorithm \ref{alg:trans}), with inputs $n$,   $(\ell_j)_{j\in \N}$, and $a$, so that $(t_{j,n})_{j\in \mI}$ is given by \eqref{eq:tjn}, with $\mI=\N$, and $(d_i)_{i\in \N}$, $\mJ$, $(H_i)_{i\in \mJ}$, and $(B_i)_{i\in \mJ}$ defined as in the proof of Theorem \ref{thm:cube}. Then, by Theorem \ref{thm:cube},  items (i) and (ii) in Theorem \ref{thm1} are satisfied, in particular $\Gamma \subset \partial \mO_\infty \subset a\overline{\mO}$.

Let us check that (iii)  in Theorem \ref{thm1} also holds, i.e.\ that $\Omega:=\R^n\setminus \Gamma$ is connected, where $\Gamma := \overline{\bigcup_{j=1}^\infty \Gamma_j}$. The construction in Theorem \ref{thm:cube} (or see Algorithm \ref{alg:trans}) is such that the cubes are arranged in layers, indexed by $\mJ$. In the case that $\mJ$ is infinite, let $H_\infty := \lim_{i\to\infty} H_i$, and note that $H_\infty\leq a$ and $x=(x_1,\ldots,x_n)\in \Omega$ if $x_n>H_\infty$. In this case it is easy to see that, for each $i\in \mJ$, $\{x\in \Omega:x_n < H_i\}$, which contains only finitely many of the $\mO_j$, is connected, so that $\{x\in \Omega:x_n < H_\infty\}$ is connected. And this set overlaps with the connected set $\Omega_E:= \{x\in \R^n: |x|>\max_{y\in \overline{\mO_\infty}}|y|\}\subset \Omega$, which in turn overlaps with the connected set $\{x\in \Omega:x_n \geq H_\infty\} \supset \{x\in \R^n:x_n>H_\infty\}$. Thus $\Omega$ is connected.

In the case that $\mJ=\{1,\ldots,M\}$ is finite and $M>1$  the $i$th layer contains finitely many $\mO_j$, if $i<M$, and the above argument gives that $\Omega':=\{x\in \Omega:x_n < H_M\} \cup \Omega_E \cup \{x\in \Omega:x_n \geq H_M+\ell_{B_M}\} \subset \Omega$ is connected, and so path-connected as $\Omega'$ is open. Further, if $x=(\tilde x,x_n)\in \Omega\setminus \Omega'$ then $H_M\leq  x_n <H_M+\ell_{B_M}$. If also $x\not\in \mO_\infty$ then the vertical line $\{(\tilde x,s): s\geq x_n\}\subset \Omega$ connects $x$ to $\Omega'$. On the other hand, if $x\in \mO_\infty\setminus \Omega'\subset \Omega$, in which case $x\in \mO_j$ for some $j\in I_M$, picking some $y\in \partial \mO_j\setminus \Gamma_j$ and some $z\in \R^n$ with $H_M-d_{M-1}<z_n<H_M$, it holds that $z\in \Omega'$ and that the polygonal path $[x,y]\cup [y,z]\subset \Omega$. Thus $\Omega$ is path-connected and so connected. The same holds by a similar but simpler argument if $M=1$.
\end{proof}

\begin{proof}[Proof of Theorem \ref{thm2v}]
Let $(\varepsilon_j)_{j\in \N}$, $(\ell_j)_{j\in \N}$, $c,\eta>0$, $M\in \N$  be as in the statement of the theorem, and $(d_j)_{j\in \N}$ be a decreasing, summable sequence of positive numbers.
For $s\in \R$, let $\lfloor s\rfloor$ denote the largest integer $\leq s$.

Working with coordinates $(x_1,\ldots,x_n)$, our methodology will be to arrange the cubes in an infinite sequence of layers parallel to the hyperplane $x_n=0$. To achieve this we choose, for each $j\in \N$, the translation $t_j=(t_{j,1},\ldots,t_{j,n})\in \R^n$ in a way we now specify. This methodology  overlaps with that of the proof of Theorem \ref{thm2}, and we will use, as far as possible, the same notations.

For $i\in \N$ let
\begin{eqnarray*}
I_i &:=& \{j\in \N: \mO_j \mbox{ is in the $i$th layer}\}, \quad B_i := \min I_i,\\ 
L_i &:=& \max\{\ell_j:j\in I_i\} = \ell_{B_i}, \quad \mbox{and} \quad N_i := \# I_i,
\end{eqnarray*}
the cardinality of $I_i$. 
We arrange the cubes so that the first cube $O_1$ is on the first level (so that $B_1=1$) 
and so that, if $O_j$ is on level $i$, for some $i,j\in \N$, then $O_{j+1}$ is either on the same level $i$ or is on level $i+1$. 
Thus we place $N_i$ cubes on the $i$th level,
\begin{equation} \label{eq:BN}
B_1 = 1 \quad \mbox{and} \quad B_{i+1} = B_i + N_i, \quad i\in \N.
\end{equation}
As in the proof of Theorem \ref{thm2} (though the index set $\mJ$ there may be finite) we set (cf.~\eqref{eq:tjn})
\begin{equation} \label{eq:tjn2}
t_{j,n} := H_i, \qquad \mbox{for } j\in I_i, \quad i\in \N,
\end{equation}
where
$$
H_1 := 0, \qquad H_{i+1} := H_i +\ell_{B_{i}}+d_{i}, \qquad i\in \N,
$$
so that
$$
H_i = \sum_{j=1}^{i-1}(\ell_{B_{j}}+d_j), \qquad i\in \N,
$$

In contrast to the proof of Theorem \ref{thm2} where the numbers of cubes that we place on each layer (and how many layers there are) depends on $(\ell_j)_{j\in \N}$, here we place $\mO_1$ on the first level by itself, so that $I_1=\{1\}$ and $N_1=B_1=1$. Then, separated from $\mO_1$ by the distance $d_1$ in the $x_n$-direction, we place the next $N_2=\lfloor 2\prod_{m=1}^{M-1}\log_m(2+\re_m)\rfloor^{n-1}$ cubes on the second level, so that $B_2=2$.
In general, we place $N_i = \lfloor i p_i\rfloor^{n-1}$ cubes on the $i$th level\footnote{A numerical calculation, using that $1\leq \log_m(x+\re_m)\leq 1+x/(\re_1\re_2\cdots\re_m)$, for $x\geq 0$ and $m\in \N$, yields that this is consistent with our specification above that $N_1=1$.}, where 
$$
p_i:= \prod_{m=1}^{M-1}\log_m(i+\re_m), \qquad i\in \N,
$$
so that
$$
B_i = 1 + \sum_{j=1}^{i-1} \lfloor j p_j\rfloor^{n-1} \quad \mbox{and} \quad 
I_i = \{B_i,\ldots,B_i+\lfloor i p_i\rfloor^{n-1}-1\}, \qquad i\in \N.
$$ 
Note that $N_i=i^{n-1}$, for $i\in \N$, if $M=1$, and that, for $i\geq 2$,
$$
B_i \geq \sum_{j=\lfloor i/2\rfloor}^{i-1} \lfloor j p_j\rfloor^{n-1} \geq \frac{i}{2}  \big\lfloor \lfloor i/2\rfloor p_{\lfloor i/2\rfloor}\big\rfloor^{n-1}
\sim \frac{i^n}{2^n} \prod_{m=1}^{M-1}\log^{n-1}_m(i)
$$
as  $i\to\infty$. Using \eqref{eq:ellbv}, this implies that for some $c^*>0$, where $\eta>0$ is as in \eqref{eq:ellbv},
\begin{equation} \label{eq:ass}
\ell_{B_i} \leq \frac{c^*}{i\log_{M}^{1+\eta}(i)\prod_{m=1}^{M-1}\log_m(i)}, \qquad \mbox{for all sufficiently large $i$}.
\end{equation}
Thus
$$
H_\infty := \lim_{i\to\infty} H_i = \sum_{j=1}^{\infty}\ell_{B_j} + \sum_{j=1}^\infty d_j
$$
is finite; the second sum converges by assumption, and the first by \eqref{eq:ass} and an application of the integral test. 

The above partial construction has a number of important features. First, the set $\mO_\infty=\bigcup_{j=1}^\infty\mO_j$ is limited in the $x_n$ direction; if $x\in \mO_\infty$, then $0 < x_n \leq H_\infty$. Second, if $\mO_j$ and $\mO_{j'}$ are on different levels, their closures do not overlap; indeed, if $j\in I_i$ and $j'\in I_{i'}$, with $1\leq i<i'$, then $\dist(\mO_j,\mO_{j'})\geq d_{i}$. Third, as long as we arrange the cubes in level $i$ so that their closures do not overlap and so that $\mO_\infty$ is bounded,  i.e. so that (i) and (ii)  in Theorem \ref{thm1} hold, then $\Omega:=\R^n\setminus \Gamma$ is connected, where $\Gamma := \overline{\bigcup_{j=1}^\infty \Gamma_j}$, i.e., (iii) in Theorem \ref{thm1} holds. Indeed, arguing as in the proof of Theorem \ref{thm2}, it is easy to see in the case that (i) and (ii) hold that, for each $i\in \N$, $\{x\in \Omega:x_n<H_i\}$, which contains only finitely many of the $\mO_j$, is connected, so that $\{x\in \Omega:x_n<H_\infty\}$ is connected. And this set overlaps with the connected set $\{x\in \R^n: |x|>\max_{y\in \overline{\mO_\infty}}|y|\}\subset \Omega$, which in turn overlaps with the connected set $\{x\in \Omega:x_n\geq H_\infty\} \supset \{x\in \R^n:x_n>H_\infty\}$. 

It remains to describe how to arrange the $N_i$ cubes $\{\mO_j:j\in I_i\}$ on level $i$, for each $i\in \N$, so that (i) and (ii) in Theorem \ref{thm1} hold. One construction is as follows. 
Recall that $L_i= \max\{\ell_j:j\in I_i\} = \ell_{B_i}$ and set $h_i := L_i + d_i/p_i$. Choose a bijection $f_i:I_i\to \{0,\ldots,\lfloor i p_i\rfloor-1\}^{n-1}$ 
and, for $j\in I_i$, set $t_j :=(\widetilde t_j, t_{j,n})$, where $t_{j,n}$ is given by \eqref{eq:tjn2} and $\widetilde t_j\in \R^{n-1}$ by $\widetilde t_j := h_i f_i(j)$. This places each cube $\mO_j$, $j\in I_i$, which has sidelength $\leq L_i$, in a larger cube of sidelength $h_i=L_i+d_i/p_i$, ensuring that $\dist(\mO_j,\mO_{j'})\geq d_i/p_i$, if $j,j'\in I_i$ with $j\neq j'$ and $i\in \N$, so that (i) in Theorem \ref{thm1} holds.
It remains to check that also (ii) in Theorem \ref{thm1} holds. Clearly, $\mO_j\subset [0,\lfloor i p_i\rfloor h_i]^{n-1}\times [0,H_\infty]$, for $j\in I_i$, $i\in \N$, so $\mO_\infty$ is bounded as long as 
$$
\sup_{i\in \N} \lfloor i p_i\rfloor h_i= \sup_{i\in \N}\lfloor i  p_i\rfloor (\ell_{B_i}+d_i/p_i)<\infty.
$$ 
Now $\lfloor i p_i\rfloor d_i/p_i \leq id_i\to 0$, as $i\to\infty$, since  $(d_i)_{i\in \N}$ is summable and decreasing (see, e.g., \cite[\S179]{Hardy}). Similarly, $\lfloor i p_i\rfloor \ell_{B_i}\to 0$, as $i\to\infty$, by \eqref{eq:ass}. Thus $\lfloor i p_i\rfloor h_i\to 0$ as $i\to\infty$, so the supremum is finite and $\mO_\infty$ is bounded. Since $\lfloor i p_i\rfloor h_i\to 0$, as $i\to\infty$, it also follows that $\Gamma\cap \{x\in \R^n:x_n=H_\infty\} = \{H_\infty e_n\}$, where $e_n$ is the unit vector in the $x_n$-direction, so that
\begin{equation} \label{eq:onep2}
\Gamma = \{H_\infty e_n\}\cup \bigcup_{j=1}^\infty \Gamma_j. 
\end{equation}
\end{proof}

\appendix
\section{Properties of Radiating Solutions}

We show in this appendix, via spherical harmonic expansions, the properties of radiating solutions of the Helmholtz equation that are captured in Lemma \ref{lem:2.1} below. This lemma was previously proved, in dimensions $n=2,3$, using similar methods of argument, in \cite[Lemma 2.1]{CWMonk2008} (and see \cite[Theorem 2.6.4]{Ned} for \eqref{eq:21one}).

The first of  \eqref{eq:21one} can be proved alternatively (in any dimension) by a straightforward integration by parts argument. (An application of Green's theorem shows that $\mathrm{Im}\, \int_{\Gamma_R} \bar u \partial_r u \, \rd s$  is independent of $R$. Then taking the limit $R\to\infty$, using the radiation condition, proves that it is non-negative.) By contrast, there appears to be no integration by parts argument to establish the second of \eqref{eq:21one}. 
The first of \eqref{eq:21one} expresses that the energy flow is outgoing, and plays a key role in energy arguments via Green's theorems. The second estimate in \eqref{eq:21one} has played an important role (in dimensions $n=2,3$) in establishing {\em a priori} estimates and positivity of sesquilinear forms (see, e.g., \cite[Lemma 3.3]{chandler2020high}). We use it in this way (for general $n\geq 2$) in our proof of \eqref{eq:coer} above.
 
 The inequality \eqref{eq:21two} was established (for $n=2,3$) in  \cite{CWMonk2008} as a component in the proof (for $n=2,3$) of the cut-off resolvent estimate Theorem \ref{thm:ss}, and has been an important component in the proofs of other similar estimates, e.g.,  \cite[Theorem 3.2]{SpKSm15} and \cite[Theorem 1.10]{chandler2020high}. The right condition on $\alpha$ for \eqref{eq:21two} to hold in dimensions $n>3$ is {\em a priori} not obvious; the condition $\alpha \geq \max(1,n-2)$ in Lemma \ref{lem:2.1} reduces to the simpler condition $\alpha \geq 1$ of \cite[Lemma 2.1]{CWMonk2008} in dimensions $n=2,3$. 
 A slightly weaker version of \eqref{eq:21two}, with the condition $\alpha \geq \max(1,n-2)$ replaced by $\alpha \geq n-1$, is proved, in the case $n=2,3$, by an integration by parts argument using a so-called {\em Morawetz multiplier}, as \cite[Lemma 2]{SpKSm15} or \cite[Lemma 2.4]{chandler2020high}. The integration by parts proofs given in \cite{SpKSm15} and \cite{chandler2020high} for $\alpha = n-1$ (the result holds for $\alpha>n-1$ if it holds for $\alpha=n-1$, by \eqref{eq:21one}), appear to generalise to dimensions $n>3$. This is of note as it is the case $\alpha = n-1$ that appears to be particularly important for applications of \eqref{eq:21two}. For example, the case $\alpha=n-1$ is enough for the proof of the cut-off resolvent estimates that are \cite[Lemma 3.5]{CWMonk2008} and \cite[Theorem 1.10]{chandler2020high} and enough for the proof of wavenumber-independent coercivity in \cite[Theorem 3.2, \S3.2]{SpKSm15}.
 
In this lemma we use the notations
$G_R:= \{x\in \R^n:|x|>R\}$ and $\Gamma_R:= \{x\in \R^n:|x|=R\}$, introduced in \S\ref{sec:weak},
$\nabla_T$ denotes the surface gradient on $\Gamma_R$ and $\nabla_S$ and $\Delta_S$ denote the surface gradient and Laplace-Beltrami operator, respectively, on the unit sphere $\mathbb{S}^{n-1}=\Gamma_1$.
\begin{lemma} \label{lem:2.1}
Suppose that $k>0$, $R>R_0>0$, and that $u\in C^2(G_{R_0})$ satisfies  $\Delta u + k^2u=0$ in $G_{R_0}$ and the Sommerfeld radiation condition \eqref{eq:src}. Then
\begin{equation} \label{eq:21one}
\mathrm{Im}\, \int_{\Gamma_R} \bar u \frac{\partial u}{\partial r} \, \rd s \geq 0, \qquad \mathrm{Re}\, \int_{\Gamma_R} \bar u \frac{\partial u}{\partial r} \, \rd s \leq 0,
\end{equation}
and, for $\alpha \geq \max(1,n-2)$,
\begin{equation} \label{eq:21two}
 \alpha\, \mathrm{Re}\, \int_{\Gamma_R} \bar u \frac{\partial u}{\partial r} \, \rd s + R\int_{\Gamma_R}\left(k^2|u|^2 +  \left|\frac{\partial u}{\partial r}\right|^2 - |\nabla_T u|^2\right) \, \rd s \leq 2kR \,\mathrm{Im}\, \int_{\Gamma_R} \bar u \frac{\partial u}{\partial r} \, \rd s.
\end{equation}
\end{lemma}

We shall prove this lemma via a reformulation in terms of Dirichlet to Neumann maps. For $R>0$ consider the exterior Dirichlet problem in $G_R$ of finding, given $\psi\in H^{1/2}(\Gamma)$, a $u\in \Hol(G_R)$ that satisfies $\Delta u + k^2u=0$ in $G_R$, the radiation condition \eqref{eq:src}, and the boundary condition $\gamma^+_R u=\psi$, where $\gamma^+_R:\Hol(G_R)\to H^{1/2}(\Gamma_R)$ is the trace operator. This problem is well-posed (e.g., \cite[Theorem 9.10]{mclean2000strongly}), so that there is a well-defined map $\DTN:H^{1/2}(\Gamma_R)\to H^{-1/2}(\Gamma_R)$ that takes the Dirichlet data $\psi$ to  $\partial_r u$, the normal derivative of the solution $u$  on $\Gamma_R$. This map $\DTN:H^{1/2}(\Gamma_R)\to H^{-1/2}(\Gamma_R)$ is continuous, indeed is also continuous as a mapping $\mathrm{DtN}_R:H^1(\Gamma_R)\to L^2(\Gamma_R)$ (see, e.g., the discussion in \cite[\S2]{spence2014wavenumber}). 

Suppose that $u$ satisfies the conditions of the above lemma. Then, for $R>R_0$, $v:= u|_{G_R}$ satisfies the above Dirichlet problem in $G_R$ with boundary data $\phi:= u|_{\Gamma_R}$, so that \eqref{eq:21one}  can be written  as
\begin{equation} \label{eq:21a}
\mathrm{Im}\, \langle \DTN \phi,\phi\rangle \geq 0, \qquad \mathrm{Re}\, \langle \DTN \phi,\phi\rangle\leq 0,
\end{equation}
where $\langle\cdot,\cdot\rangle$ denotes the duality pairing on $H^{-1/2}(\Gamma_R)\times H^{1/2}(\Gamma_R)$ that is an extension of  the inner product $(\cdot,\cdot)_{L^2(\Gamma_R)}$ on $L^2(\Gamma_R)$, 
and \eqref{eq:21two} can be written as
\begin{eqnarray} \nonumber
 \alpha\, \mathrm{Re}\, \langle \DTN \phi,\phi\rangle + R\left(k^2\|\phi\|_{L^2(\Gamma_R)}^2 +  \|\DTN\phi\|_{L^2(\Gamma_R)}^2 - \|\nabla_T \phi\|_{L^2(\Gamma_R)}^2\right)\\ \label{eq:21b}
 \hspace{10ex}  \leq 2kR \,\mathrm{Im}\, \langle \DTN \phi,\phi\rangle.
\end{eqnarray}
Thus Lemma \ref{lem:2.1} is implied by the following rephrased result.

\begin{lemma} \label{lem:2.1b}
Suppose that $R>0$ and $k>0$. Then \eqref{eq:21a} holds for all $\phi\in H^{1/2}(\Gamma_R)$ and, if $\alpha \geq \max(1,n-2)$, then also   \eqref{eq:21b} holds for all $\phi\in H^{1}(\Gamma_R)$.
\end{lemma}
\begin{proof} 
Since $\DTN:H^{1/2}(\Gamma_R)\to H^{-1/2}(\Gamma_R)$ and $\DTN:H^1(\Gamma_R)\to L^2(\Gamma_R)$ are continuous, it is enough to prove that these inequalities hold for all $\phi\in \cH_R$, for some $\cH_R\subset H^1(\Gamma_R)$ that is dense in $H^1(\Gamma_R)$ (and so also in $H^{1/2}(\Gamma_R)$). We will choose a set $\cH_R$ defined in terms of spherical harmonics, so first recall  relevant notations and results.

 For $m\in \N_0:= \N\cup\{0\}$, let $\cH^m(\mathbb{S}^{n-1})$ denote the set of spherical harmonics of degree $m$ on $\mathbb{S}^{n-1}$, as defined, e.g., in \cite[p.~252]{mclean2000strongly}. Note (see, e.g., \cite[Theorem 1.2.16]{levitin2023topics} and its proof for these details) that $L^2(\mathbb{S}^{n-1})$ has the orthogonal decomposition 
 $$
 L^2(\mathbb{S}^{n-1}) = \bigoplus_{m=0}^\infty \cH^m(\mathbb{S}^{n-1}),
 $$ 
 and the spherical harmonics are the eigenfunctions of $-\Delta_S$; precisely, each $\cH^m(\mathbb{S}^{n-1})$ is a finite-dimensional eigenspace
 corresponding to the eigenvalue $m(m+n-2)$. Thus every $\phi\in  L^2(\mathbb{S}^{n-1})$ has an expansion 
 \begin{equation} \label{eq:expan}
 \phi = \sum_{m=0}^\infty \phi_m \quad \mbox{with} \quad  \|\phi\|_{  L^2(\mathbb{S}^{n-1})}^2 = \sum_{m=0}^\infty \|\phi_m\|_{  L^2(\mathbb{S}^{n-1})}^2,
 \end{equation}
 where $\phi_m$ is the orthogonal projection of $u$ from $L^2(\mathbb{S}^{n-1})$ to $\cH^m(\mathbb{S}^{n-1})$. Let $\cH\subset L^2(\mathbb{S}^{n-1})$ denote the set of spherical harmonics, by which we mean the set of those $\phi\in L^2(\mathbb{S}^{n-1})$ that have a representation \eqref{eq:expan} with only finitely many terms, i.e.\ with $\phi_m=0$ for all but finitely many $m$. Clearly, $\cH$ is a dense subspace of $L^2(\mathbb{S}^{n-1})$. Let us show that $\cH$ is also dense in $H^1(\mathbb{S}^{n-1})$. 
 
 Let $(\cdot,\cdot)_S$ denote the inner product on  $L^2(\mathbb{S}^{n-1})$. 
 Recall (e.g., \cite[Eq.~(1.2.2)]{levitin2023topics}) that
 \begin{equation} \label{eq:div}
\int_{\mathbb{S}^{n-1}} \phi \Delta_S \psi \,\rd s = -  \int_{\mathbb{S}^{n-1}} \nabla_S\phi \cdot \nabla_S \psi \,\rd s, \qquad \psi\in H^2(\mathbb{S}^{n-1}), \quad \phi\in H^1(S).
 \end{equation}
 If  $\phi\in  H^1(\mathbb{S}^{n-1})$ has the expansion \eqref{eq:expan} then, by \eqref{eq:div}, for $i,j\in \N_0$,
 \begin{eqnarray} \nonumber
 \int_{\mathbb{S}^{n-1}} \nabla_S \phi_i\cdot \nabla_S \bar \psi_j\, \rd s&=&-(\Delta_S\phi_i,\phi_j)_S\\ \label{eq:or1}
 &=& i(i+n-2) (\phi_i,\phi_j)_S = i(i+n-2)\|\phi_i\|^2_{ L^2(\mathbb{S}^{n-1})}\delta_{ij},
 \end{eqnarray}
 where $\delta_{ij}$ is the Kronecker delta. In particular,
 \begin{equation} \label{eq:or2}
 \|\nabla_S \phi_m\|_{L^2(\mathbb{S}^{n-1})}^2 = m(m+n-2)\|\phi_m\|^2_{ L^2(\mathbb{S}^{n-1})}, \qquad m\in \N_0.
 \end{equation}
 Further, for $m\in \N_0$,
 \begin{eqnarray*}
  m(m+n-2)\|\phi_m\|^2_{L^2(\mathbb{S}^{n-1})}  &=& m(m+n-2) \left|(\phi_m,\phi)_S\right| = \left|(\Delta_S\phi_m,\phi)_S\right|\\ 
  &=& \left|\int_{\mathbb{S}^{n-1}} \nabla_S \phi_m\cdot \nabla_S \bar \phi\, \rd s\right|\\
  & \leq & \|\nabla_S \phi_m\|_{L^2(\mathbb{S}^{n-1})}\|\nabla_S \phi\|_{L^2(\mathbb{S}^{n-1})}\\& \leq &\sqrt{m(m+n-2)}\, \|\phi_m\|_{L^2(\mathbb{S}^{n-1})}\|\nabla_S \phi\|_{L^2(\mathbb{S}^{n-1})},
 \end{eqnarray*}
 so that
 \begin{equation} \label{eq:bd}
 \|\phi_m\|_{L^2(\mathbb{S}^{n-1})}  \leq m^{-1} \|\nabla_S \phi\|_{L^2(\mathbb{S}^{n-1})},\qquad m\in \N.
 \end{equation}
 It follows from \eqref{eq:or1}, \eqref{eq:or2}, and \eqref{eq:bd} 
 that the series \eqref{eq:expan} also converges in $H^1(\mathbb{S}^{n-1})$, so that $\cH$ is also dense in $H^1(\mathbb{S}^{n-1})$, and that
 \begin{equation} \label{eq:nabla2}
 \|\nabla_S \phi\|_{  L^2(\mathbb{S}^{n-1})}^2 = \sum_{m=0}^\infty \|\nabla_S \phi_m\|_{  L^2(\mathbb{S}^{n-1})}^2=m(m+n-2)\sum_{m=0}^\infty \|\phi_m\|_{  L^2(\mathbb{S}^{n-1})}^2.
 \end{equation}
 
 
For $\psi \in L^1(\Gamma_R)$ define $\widehat \psi\in L^1(\mathbb{S}^{n-1})$ by $\widehat \psi(\omega)=\psi(R\omega)$, $\omega\in \mathbb{S}^{n-1}$, and note that
\begin{equation} \label{eq:trans}
\int_{\Gamma_R}\psi \,\rd s = R^{n-1} \int_{\mathbb{S}^{n-1}}\widehat\psi \,\rd s, \qquad \psi\in L^1(\Gamma_R).
\end{equation}
Let $\cH_R:= \{\psi\in L^2 (\Gamma_R):\widehat \psi\in \cH\}$. The density of $\cH$ in $H^1(\mathbb{S}^{n-1})$ implies that $\cH_R\subset H^1(\Gamma_R)$ is dense in $H^1(\Gamma_R)$. We complete the proof of the lemma by showing that \eqref{eq:21a} and \eqref{eq:21b} hold for all $\phi\in \cH_R$.

 For $\nu\geq0$ let $J_\nu$ and $Y_\nu$ denote the usual Bessel functions of the first and second kinds of order $\nu$, and let $H_\nu^{(1)}:= J_\nu + \ri Y_\nu$, so that $H_\nu^{(1)}$ is the Hankel function of the first kind of order $\nu$. For $m\in \N_0$ and $t>0$, let
 $$
  j_m(t) = j_m(n,t) := \frac{J_{m+n/2-1}(t)}{t^{n/2-1}}, \quad  y_m(t) = y_m(n,t):= \frac{Y_{m+n/2-1}(t)}{t^{n/2-1}},
 $$
 so that $j_m$ and $y_m$ are spherical Bessel functions of the first and second kinds. (Our notation is that of \cite{mclean2000strongly} or (for the case $n=3$) \cite{Ned,NIST}, except that we omit a normalising factor $\sqrt{\pi/2}$.) For $m\in \N_0$ and $t>0$ define also the associated spherical Hankel function of the first kind,
 $$
 h_m(t) = h_m(n,t) := j_m(n,t)+\ri y_m(n,t) = \frac{H^{(1)}_{m+n/2-1}(t)}{t^{n/2-1}}.
 $$
 It is convenient to also use the notations (e.g., \cite[\S10.18]{NIST})
$$
M_\nu(t) := |H_\nu^{(1)}(t)|, \quad N_\nu(t) := \left|{H_\nu^{(1)}}^\prime(t)\right|, \qquad t>0, \quad \nu\geq 0.
$$ 
The arguments we make below use that $M_\nu$ is decreasing on $(0,\infty)$  for $\nu\geq 0$ \cite[\S13.74]{Watson}, so that also  $|h_m|$ is decreasing on $(0,\infty)$, for $m\in \N_0$.
If $\psi\in \cH^m(\mathbb{S}^{n-1})$, then $\psi(x/|x|)h_m(n,k|x|)$ is a solution of the Helmholtz equation in $\R^n\setminus\{0\}$ that satisfies the Sommerfeld radiation condition (see \cite[Lemma 9.3, p.~282]{mclean2000strongly}).

Suppose $\phi\in \cH_R$, so that $\widehat \phi\in \cH$, and let $\phi_m$ denote the component of $\widehat \phi$ in $\cH^m(\mathbb{S}^{n-1})$, for $m\in \N_0$, noting that only finitely many of the $\phi_m$ are non-zero. Then
\begin{equation} \label{eq:onr0}
\phi(x) = \sum_{m=0}^\infty \phi_{m}(x/|x|), \qquad x\in \Gamma_{R},
\end{equation} 
so that 
\begin{equation} \label{eq:onr3}
\nabla_T\phi(x) = R^{-1}\sum_{m=0}^\infty \nabla_S\phi_{m}(x/|x|), \qquad x\in \Gamma_{R}.
\end{equation} 
By inspection, the solution of the Dirichlet problem in $G_R$ with boundary data $\phi$ is, where $\rho:=kR$,
\begin{equation} \label{eq:onr}
u(x) = \sum_{m=0}^\infty \phi_{m}(x/|x|)h_m(k|x|)/h_m(\rho), \qquad x\in G_{R},
\end{equation}
so that
\begin{equation} \label{eq:onr2}
\DTN\phi(x) = k\sum_{m=0}^\infty \phi_{m}(x/|x|)h'_m(\rho)/h_m(\rho), \qquad x\in \Gamma_R.
\end{equation}
From \eqref{eq:trans}, \eqref{eq:onr0}, and \eqref{eq:onr2} we see that, where $c_m:= \|\phi_m\|_{  L^2(\mathbb{S}^{n-1})}^2/|h_m(\rho)|^2$,
\begin{eqnarray*}
\langle \DTN \phi,\phi\rangle &=& R^{n-2}\rho\sum_{m=0}^\infty c_m h'_m(\rho)\overline{h_m(\rho)}\\ & =& R^{n-2}\rho\sum_{m=0}^\infty c_m \left(\frac{1}{2}\frac{\rd}{\rd \rho}\left(|h_m(\rho)|^2\right)+\frac{2\ri}{\pi \rho^{n-1}}\right),
\end{eqnarray*}
where to obtain this last result we have use the Wronskian formula that $j_m(\rho)y'_m(\rho)-y_m(\rho) j'_m(\rho)=2\rho^{1-n}/\pi$ (e.g., \cite[Exercise~9.3]{mclean2000strongly}). This establishes \eqref{eq:21a} since $|h_m|$ is decreasing on $(0,\infty)$.

Similarly, using \eqref{eq:trans}, \eqref{eq:onr0}, \eqref{eq:onr3}, \eqref{eq:onr2}, and \eqref{eq:nabla2}, we see that
\begin{eqnarray*}
& & R\left(k^2\|\phi\|_{L^2(\Gamma_R)}^2 +  \|\DTN\phi\|_{L^2(\Gamma_R)}^2 - \|\nabla_T \phi\|_{L^2(\Gamma_R)}^2\right)\\ &=&  R^{n-2}\sum_{m=0}^\infty c_m\left((\rho^2-m(m+2-2))|h_m(\rho)|^2 +\rho^2|h'_m(\rho)|^2\right).
\end{eqnarray*}
Thus \eqref{eq:21b} holds for all $\phi\in \cH_R$ if and only if $B_m(\rho)\leq 0$ for $m\in \N_0$ and $\rho>0$, where
\begin{equation} \label{eq:Bm}
B_m(\rho) := (\rho^2-m(m+2-2))|h_m(\rho)|^2 +\rho^2|h'_m(\rho)|^2 + \frac{\alpha \rho}{2}\frac{\rd}{\rd \rho}\left(|h_m(\rho)|^2\right)-\frac{4}{\pi}\rho^{3-n}.
\end{equation}

To see that $B_m(\rho)\leq 0$ if $\alpha \geq \max(1,n-2)$, completing the proof, let $\nu = m+n/2-1$, $p=1-n/2$, so that $m(m+n-2) = \nu^2-p^2$, $h_m(\rho)=\rho^pH_\nu^{(1)}(\rho)$,
\begin{eqnarray*}
h_m'(\rho) &=&\rho^{p-1}(\rho{H_\nu^{(1)}}^\prime(\rho) + pH_\nu^{(1)}(\rho)),\\
\rho^n |h_m'(\rho)|^2 &=&  \rho^2N^2_\nu(\rho) + p^2 M^2_\nu(\rho)+p\rho\,\frac{\rd}{\rd \rho}(M_\nu^2(\rho)),
\end{eqnarray*}
and
$$
\frac{\rho^{n-1}}{2}\frac{\rd}{\rd \rho}(|h_m(\rho)|^2)  = pM_\nu^2(\rho) + \frac{\rho}{2}\frac{\rd}{\rd \rho}(M^2_\nu(\rho)).
$$
Using these identities it follows that
$$
\rho^{n-2}B_m(\rho) =  (\rho^2-\nu^2+2p^2+\alpha p)M^2_\nu(\rho) +\rho^2N_\nu^2(\rho) +\left(p+\frac{\alpha}{2}\right)\rho\,\frac{\rd}{\rd \rho}(M_\nu^2(\rho)) -\frac{4\rho}{\pi}, 
$$
for $m\in \N_0$ and $\rho>0$. For $\nu\geq 0$ and $t>0$, let
$$
A_\nu(t) :=M_\nu^2(t)(t^2-\nu^2)+t^2N_\nu^2(t)-4t/\pi.
$$
Then, for $m\in \N_0$ and $\rho>0$,
\begin{eqnarray*}
\rho^{n-2}B_m(\rho) &=&  A_\nu(\rho) +\frac{2-n+\alpha}{2}\left((2-n)M^2_\nu(\rho) +\rho \frac{\rd}{\rd \rho}(M_\nu^2(\rho))\right)\\ & \leq  & A_\nu(\rho) +\frac{2-n+\alpha}{2}\frac{\rd}{\rd \rho}(M_\nu^2(\rho)),
\end{eqnarray*}
if $\alpha\geq n-2$. Now it is shown in \cite[Equation (2.7)]{CWMonk2008} that $A_\nu(t)\leq 0$, for $\nu\geq 1/2$ and $t>0$. Since also $M_\nu$ is decreasing on $(0,\infty)$, it follows that $B_m(\rho)\leq 0$ for $m\in \N_0$ and $\rho>0$, if $\alpha\geq n-2$, as long as $n\geq 3$. In the case $n=2$,
$$
B_m(\rho) =  A_m(\rho) +\frac{\alpha\rho}{2}\frac{\rd}{\rd \rho}(M_m^2(\rho)),
$$
for $m\in \N_0$ and $\rho>0$, so that $\alpha\geq n-2=0$ implies that $B_m(\rho)\leq 0$ for $\rho>0$ if $m\in \N$. 
If $\alpha \geq 1$, then also
$$
B_0(\rho) =  A_0(\rho) +\frac{\alpha}{2}\frac{\rd}{\rd \rho}(M_0^2(\rho)) \leq A_0(\rho) +\frac{1}{2}\frac{\rd}{\rd \rho}(M_0^2(\rho)) \leq 0,
$$
for $\rho>0$, see \cite[pp.~1436--1437]{CWMonk2008}.
\end{proof}

\subsection*{Acknowledgements} The authors thank Euan Spence (Bath) for useful discussions, including drawing our attention to reference \cite{ChFr25}, and thank the anonymous referee and associate editor for helpful comments and feedback, that have fed, in particular, into the discussion before the proof of Theorem \ref{thm1} in \S\ref{sec:thm1} and into refinements of the proof of Theorem \ref{thm2v}. We acknowledge, in the preparation of this second revision of the manuscript, use of ChatGPT running GPT-5.5 to: i) assist in an additional  literature search, leading to the discovery of \cite{MM68}; ii) assist in checking the accuracy of the additional arguments and calculations that have been added in this second revision. The authors assume responsibility for all content.


\begin{thebibliography}{10}

\bibitem{BetCha11}
{\sc T.~Betcke, S.~N. Chandler-Wilde, I.~G. Graham, S.~Langdon, and
  M.~Lindner}, {\em {Condition number estimates for combined potential boundary
  integral operators in acoustics and their boundary element discretisation}},
  Numer. Methods Partial Differential Eq., 27 (2011), pp.~31--69.

\bibitem{BrSc:08}
{\sc S.~C. Brenner and L.~R. Scott}, {\em The Mathematical Theory of Finite
  Element Methods, 3rd Ed.}, Springer, 2008.

\bibitem{Bruno}
{\sc O.~P. Bruno, M. Santana, and L.~N.~Trefethen}, {\em Evaluation of resonances: adaptivity and AAA
rational approximation of randomly scalarized boundary integral resolvents}, SIAM J. Sci. Comput., 48 (2026), pp.~A1260--A1283.


\bibitem{Burq1998}
{\sc N.~Burq}, {\em D\'ecroissance de l'\'energie locale de l’\'equation des
  ondes pour le probl\`eme ext\'erieur et absence de r\'esonance au voisinage
  du r\'eel}, Acta Math., 180 (1998), pp.~1--29.

\bibitem{Bu:04}
{\sc N.~Burq}, {\em Smoothing effect for {S}chr\"{o}dinger boundary value
  problems}, Duke Math. J., 123 (2004), pp.~403--427.

\bibitem{CaPo:02}
{\sc F.~Cardoso and G.~Popov}, {\em Quasimodes with exponentially small errors
  associated with elliptic periodic rays}, Asymptot. Anal., 30 (2002),
  pp.~217--247.

\bibitem{chandler2012numerical}
{\sc S.~N. Chandler-Wilde, I.~G. Graham, S.~Langdon, and E.~A. Spence}, {\em
  Numerical-asymptotic boundary integral methods in high-frequency acoustic
  scattering}, Acta Numer., 21 (2012), pp.~89--305.

\bibitem{chandler2018well}
{\sc S.~N. Chandler-Wilde and D.~P. Hewett}, {\em Well-posed {PDE} and integral
  equation formulations for scattering by fractal screens}, SIAM J. Math.
  Anal., 50 (2018), pp.~677--717.

\bibitem{CWHeMoBe:21}
{\sc S.~N. Chandler-Wilde, D.~P. Hewett, A.~Moiola, and J.~Besson}, {\em
  Boundary element methods for acoustic scattering by fractal screens}, Numer.
  Math., 147 (2021), pp.~785--837.

\bibitem{CWMonk2008}
{\sc S.~N. Chandler-Wilde and P.~Monk}, {\em Wave-number-explicit bounds in
  time-harmonic scattering}, SIAM J. Math. Anal., 39 (2008), pp.~1428--1455.

\bibitem{SiavashSimon0}
{\sc S.~N. Chandler-Wilde and S.~Sadeghi},
{\em Integral equation methods for scattering by general compact obstacles: wavenumber-explicit estimates},
Preprint arXiv:2601.19456, (2025).

\bibitem{SiavashSimon1}
{\sc S.~N. Chandler-Wilde and S.~Sadeghi}, {\em Wavenumber-explicit bounds for
  resolvents and first kind integral equations in time-harmonic scattering},
  2025.
\newblock In preparation.

\bibitem{chandler2020high}
{\sc S.~N. Chandler-Wilde, E.~A. Spence, A.~Gibbs, and V.~P. Smyshlyaev}, {\em
  High-frequency bounds for the {H}elmholtz equation under parabolic trapping
  and applications in numerical analysis}, SIAM J. Math. Anal., 52 (2020),
  pp.~845--893.

\bibitem{ChFr25}
{\sc T.~Chaumont-Frelet}, {\em A ill-posed scattering problem saturating
  {W}eyl's law}, Preprint arXiv:2505.10346,  (2025).

\bibitem{DyZw}
{\sc S.~Dyatlov and M.~Zworski}, {\em Mathematical Theory of Scattering Resonances}, American Mathematical Society, 2019.


\bibitem{Gady}
{\sc R.~R. Gadyl'shin}, {\em Existence and asymptotics of poles with small imaginary
part for the Helmholtz resonator}, Russ. Math. Surv., 52 (1997), pp.~1--72.

\bibitem{Gady2}
{\sc R.~R. Gadyl'shin}, {\em On the Dirichlet problem for the Helmholtz
equation on the plane with boundary conditions on
an almost closed curve},
Sb. Math. 191 (2000), pp.~821--848.

\bibitem{Grisvard}
{\sc P.~Grisvard}, {\em Elliptic Problems in Nonsmooth Domains}, Pitman,
  Boston, 1985.

\bibitem{Hardy}
{\sc G.~H. Hardy}, {\em A Course of Pure Mathematics, 9th Ed.}, Cambridge
  University Press, 1944.

\bibitem{Ih:98}
{\sc F.~Ihlenburg}, {\em Finite Element Analysis of Acoustic Scattering},
  Springer Verlag, 1998.

\bibitem{lafontaine2021most}
{\sc D.~Lafontaine, E.~A. Spence, and J.~Wunsch}, {\em For most frequencies,
  strong trapping has a weak effect in frequency-domain scattering}, Comm. Pure
  Appl. Math., 74 (2021), pp.~2025--2063.

\bibitem{LaPh:89}
{\sc P.~D. Lax and R.~S. Phillips}, {\em Scattering Theory}, Academic Press,
  Boston, 2nd~ed., 1989.

\bibitem{levitin2023topics}
{\sc M.~Levitin, D.~Mangoubi, and I.~Polterovich}, {\em Topics in Spectral
  Geometry}, American Mathematical Society, 2023.

\bibitem{mclean2000strongly}
{\sc W.~C.~H. McLean}, {\em Strongly Elliptic Systems and Boundary Integral
  Equations}, Cambridge University Press, 2000.
  
 \bibitem{MM68} 
{\sc A. Meir and l. Moser}, {\em On packing of squares and cubes},  J. Comb. Theory, 5 (1968), pp.~126--134.


\bibitem{Me:79}
{\sc R.~B. Melrose}, {\em Singularities and energy decay in acoustical
  scattering}, Duke Math. J., 46 (1979), pp.~43--59.

\bibitem{MeSj:78}
{\sc R.~B. Melrose and J.~Sj{\"o}strand}, {\em {Singularities of boundary value
  problems. I}}, Comm. Pure Appl. Math., 31 (1978), pp.~593--617.

\bibitem{MeSj:82}
{\sc R.~B. Melrose and J.~Sj{\"o}strand}, {\em {Singularities of boundary value
  problems. II}}, Comm. Pure Appl. Math., 35 (1982), pp.~129--168.

\bibitem{morawetz1975}
{\sc C.~S. Morawetz}, {\em Decay for solutions of the exterior problem for the
  wave equation}, Comm. Pure Appl. Math., 28 (1975), pp.~229--264.

\bibitem{morawetzludwig1968}
{\sc C.~S. Morawetz and D.~Ludwig}, {\em An inequality for the reduced wave
  operator and the justification of geometrical optics}, Comm. Pure Appl.
  Math., 21 (1968), pp.~187--203.

\bibitem{MoRaSt:77}
{\sc C.~S. Morawetz, J.~V. Ralston, and W.~A. Strauss}, {\em {Decay of
  solutions of the wave equation outside nontrapping obstacles}}, Comm. Pure
  Appl. Math., 30 (1977), pp.~447--508.

\bibitem{Ned}
{\sc J.~C. N{\'e}d{\'e}lec}, {\em {Acoustic and Electromagnetic Equations:
  Integral Representations for Harmonic Problems}}, Springer Verlag, 2001.

\bibitem{NR87}
{\sc P.~Neittaanm\"aki and G.~F. Roach}, {\em Weighted {S}obolev spaces and
  exterior problems for the {H}elmholtz equation}, Proc. R. Soc. Lond. A, 410
  (1987), pp.~373--383.

\bibitem{NIST}
{\sc NIST}, {\em Digital library of mathematical functions}, 2023.
\newblock See \url{https://dlmf.nist.gov}.

\bibitem{Polking72}
{\sc J.~C. Polking}, {\em Leibniz formula for some differentiation operators of
  fractional order}, Indiana U. Math. J., 21 (1972), pp.~1019--1029.

\bibitem{Po:91}
{\sc G.~S. Popov}, {\em {Quasimodes for the Laplace operator and glancing
  hypersurfaces}}, in Microlocal Analysis and Nonlinear Waves, Springer, 1991,
  pp.~167--178.

\bibitem{Ra:79}
{\sc J.~Ralston}, {\em Note on the decay of acoustic waves}, Duke Math. J., 46
  (1979), pp.~799--804.
  
\bibitem{SaSch}
{\sc S.~A. Sauter and C.~Schwab}, {\em Boundary Element Methods},
  Springer, 2011.
  
 \bibitem{SZ94}
 {\sc J.~Sj\"ostrand and M.~Zworski}, {\em Lower bounds on the number of scattering poles. II}, J. Funct. Anal., 123 (1994), pp.~336--367.
 

\bibitem{spence2014wavenumber}
{\sc E.~A. Spence}, {\em Wavenumber-explicit bounds in time-harmonic acoustic
  scattering}, SIAM J. Math. Anal., 46 (2014), pp.~2987--3024.

\bibitem{Sp23}
{\sc E.~A. Spence}, {\em A simple proof that the $hp$-{FEM} does not suffer
  from the pollution effect for the constant-coefficient full-space {H}elmholtz
  equation}, Adv. Comp. Math., 49 (2023), p.~article number 27.

\bibitem{SpKSm15}
{\sc E.~A. Spence, I.~V. Kamotski, and V.~P. Smyshlyaev}, {\em Coercivity of
  combined boundary integral equations in high-frequency scattering}, Comm.
  Pure Appl. Math., 68 (2015), pp.~1587--1639.


\bibitem{Va:75}
{\sc B.~R. Vainberg}, {\em {On the short wave asymptotic behaviour of solutions
  of stationary problems and the asymptotic behaviour as $t\rightarrow \infty$
  of solutions of non-stationary problems}}, Russ. Math. Surv., 30 (1975),
  pp.~1--58.

\bibitem{Watson}
{\sc G.~N. Watson}, {\em Theory of {B}essel Functions}, Cambridge University
  Press, 1922.

\bibitem{wilcox1975}
{\sc C.~H. Wilcox}, {\em Scattering Theory for the d'{A}lembert Equation in
  Exterior Domains}, Springer, 1975.

\end{thebibliography}

\end{document}